\documentclass[12pt, leqno]{article}
\usepackage{amsmath}
\usepackage{amsthm}
\usepackage{amssymb}
\usepackage{latexsym}
\usepackage{graphicx}
\allowdisplaybreaks[1]
\usepackage{amsfonts}
\usepackage{ascmac}
\usepackage{color}
\usepackage{enumerate}

\theoremstyle{definition}
\theoremstyle{plain}

\allowdisplaybreaks[1]

\date{}
\newtheorem{Thm}{Theorem}[section]
\newtheorem{Prop}[Thm]{Proposition}
\newtheorem{Lemma}[Thm]{Lemma}

\newcommand{\bu}{\bar{u}}

\newcommand{\dt}{\Delta t}
\newcommand{\dx}{\Delta x}
\newcommand{\I}{\mathcal{I}}

\newcommand{\mes}{{\rm meas\,}}

\newcommand{\dis}{\displaystyle}

\newcommand{\norm}{\parallel}

\newcommand{\Z}{{\mathbb Z}}

\newcommand{\T}{{\mathbb T}}
\newcommand{\N}{{\mathbb N}}
\newcommand{\R}{{\mathbb R}}

\def\text#1{\mbox{#1 }}

\DeclareMathOperator{\esssup}{ess \sup}
\title{\bf More on Stochastic and Variational Approach to the Lax-Friedrichs Scheme}
\author{Kohei Soga
\footnote{Department of Pure and Applied Mathematics, Waseda University, Tokyo 169-8555, Japan \qquad\qquad
(kohei-math@toki.waseda.jp).}}
\begin{document}
\maketitle
\begin{abstract}
\noindent A stochastic and variational aspect of the Lax-Friedrichs scheme was applied to hyperbolic scalar conservation laws by Soga [arXiv: 1205.2167v1].
The results for the Lax-Friedrichs scheme are extended here to show its time-global stability, the large-time behavior, and error estimates.
The proofs essentially rely on the calculus of variations in the Lax-Friedrichs scheme and on the theory of viscosity solutions of Hamilton-Jacobi equations corresponding to the hyperbolic scalar conservation laws.
Also provided are basic facts that are useful in the numerical analysis and simulation of the weak Kolmogorov-Arnold-Moser (KAM) theory.
As one application, a finite difference approximation to KAM tori is rigorously treated.
\noindent

\medskip

  \medskip

\end{abstract}
\setcounter{section}{0}
\setcounter{equation}{0}
\section{Introduction}
We investigate the Lax-Friedrichs scheme applied to initial value problems of hyperbolic scalar conservation laws with a constant $c$,
\begin{eqnarray}\label{CL2}
u_t+H(x,t,c+u)_x=0.
\end{eqnarray}
There is a vast literature on the stability and convergence of the scheme.
The standard technique is based on the $L^1$-framework with a priori estimates and the compactness of functions of bounded variation, where mesh-size independent boundedness of both the difference solutions and their total variation must be verified.
Since the Lax-Friedrichs scheme is very simple, details of approximation can be successfully analyzed, particularly in the case of a flux function with the simple form $H(x,t,p)=H(p)$.
We refer to \cite{Crandall-Majda}, \cite{Sabac}, \cite{Tadmor}, and the studies cited therein.
However, in the case of a general flux function depending on both $x$ and $t$, the problem becomes far more difficult and often requires undesirable assumptions.
The results of the general case first appeared in~\cite{Oleinik}, where stability and $L^1$-convergence are proved with a restricted time interval that is determined by the growth of $H(x,t,p)$ with respect to $p$.
In~\cite{Nishida-Soga}, time-global stability and $L^1$-convergence within arbitrary time intervals are proved for a flux function of the form $H(x,t,p)=f(p)+F(x,t)$ in the periodic setting, with many details of the large-time behavior of the Lax-Friedrichs scheme.
Still, it seems very difficult to obtain results similar to those in~\cite{Nishida-Soga} for more general flux functions by the standard approach based on the $L^1$-framework.

Recently, a stochastic and variational approach to the Lax-Friedrichs scheme was announced \cite{Soga2}.
Stability and convergence were proved on the basis of 1) the {\it law of large numbers} in the hyperbolic scaling limit of random walks, and 2) the {\it calculus of variations} in the theory of viscosity solutions of the Hamilton-Jacobi equations with constants $c$ and $h(c )$,
\begin{eqnarray}\label{HJ2}
v_t+H(x,t,c+v_x)=h(c ).
\end{eqnarray}
This is a finite difference version of the stochastic and variational approach to the vanishing viscosity method in~\cite{Fleming}.
Now we briefly review the stochastic and variational approach in~\cite{Soga2}.
Consider initial value problems of the inviscid hyperbolic scalar conservation law and the corresponding Hamilton-Jacobi equation
\begin{eqnarray}\label{CL}
\left\{
\begin{array}{lll}
&\dis u_t+H(x,t,c+u)_x=0\,\,\,\,\mbox{in $\T\times(0,T]$,}\medskip\\
&\dis u(x,0)=u^0(x)\in L^\infty(\T)\,\,\,\,\mbox{on $\T$},\quad\int_\T u^0(x)dx=0,\quad \norm u^0\norm_{L^\infty}\le r ,
\end{array}
\right.
\end{eqnarray}
\begin{eqnarray}\label{HJ}
\left\{
\begin{array}{lll}
&\dis v_t+H(x,t,c+v_x)=h(c)\,\,\,\,\mbox{in $\T\times(0,T]$,}\medskip\\
&v(x,0)=v^0(x)\in Lip(\T)\,\,\,\,\mbox{on $\T$}, \quad\norm v_x^0\norm_{L^\infty}\le r,
\end{array}
\right.
\end{eqnarray}
where $c\in [c_0,c_1]$ is a varying parameter, $\T:=\R/\Z$ is the standard torus, $h(c )$ is a continuous function, and $r>0$ is a constant.
We arbitrarily fix $T$, $r$, and $[c_0,c_1]$.
Note that (\ref{CL}) and (\ref{HJ}) are equivalent in the sense that the entropy solution $u$ or viscosity solution $v$ is derived from the other if $u^0=v^0_x$.
In particular, we have  $u=v_x$ (e.g., see \cite{Bernard}).
Hereinafter we assume that $u^0=v^0_x$.
The flux function $H$ is assumed to satisfy the following:

(A1) $H(x,t,p):\T^2\times\R\to\R$, $C^2$ \quad
(A2) $H_{pp}>0$  \quad
(A3) $\dis \lim_{|p|\to+\infty}\frac{H(x,t,p)}{|p|}=+\infty$.

\noindent From (A1)--(A3) we obtain the Legendre transform $L(x,t,\xi)$ of $H(x,t,\cdot)$, which is given by
 $$L(x,t,\xi)=\sup_{p\in\R}\{\xi p -H(x,t,p)\}$$
and satisfies

(A1)$'$ $L(x,t,\xi):\T^2\times\R\to\R$, $C^2$ \quad
(A2)$'$ $L_{\xi\xi}>0$  \quad
(A3)$'$ $\dis \lim_{|\xi|\to+\infty}\frac{L(x,t,\xi)}{|\xi|}=+\infty$.

\noindent The final assumption is the following:

 (A4) There exists $\alpha>0$ such that $|L_x|\le\alpha(|L|+1)$.

\noindent  Note that (A4) implies completeness of the Euler-Lagrange flow generated by $L$ and Hamiltonian flow generated by $H$.

We discretize the equation in (\ref{CL}) by the Lax-Friedrichs scheme as follows:
\begin{eqnarray}\label{CL-Delta}
 \frac{u^{k+1}_{m+1}-\frac{(u^k_{m}+u^k_{m+2})}{2}}{\dt}
+\frac{H(x_{m+2},t_k,c+u^k_{m+2})-H(x_m,t_k,c+u^k_m)}{2\dx}
=0.
\end{eqnarray}
We also discretize the equation in (\ref{HJ}) by the following scheme:
\begin{eqnarray}\label{HJ-Delta}\frac{ v^{k+1}_{m} - \frac{(v^k_{m-1}+v^k_{m+1})}{2} }{\dt}+H(x_{m},t_k,c+\frac{v^k_{m+1}-v^k_{m-1}}{2\dx})
=h(c).
\end{eqnarray}
Note that (\ref{CL-Delta}) and (\ref{HJ-Delta}) are also equivalent in the sense that $u^k_m$ or $v^k_{m+1}$ is derived from the other.
In particular, we have
$$u^k_m=\frac{v^k_{m+1}-v^k_{m-1}}{2\dx},$$
which is an important relation in this paper.
 In the stochastic and variational approach, the {\it stochastic} comes from the numerical viscosity intrinsic to (\ref{CL-Delta}) and (\ref{HJ-Delta}), while the {\it variational} comes from the variational structures of Hamilton-Jacobi equations.
The stochastic and variational approach in~\cite{Soga2} led to several results:
\begin{enumerate}
 \item[(1)] Stochastic and variational representation formulas (value functions) for  $v^k_{m+1}$ and $u^k_m$ were obtained.
 \item[(2)] Stability of the Lax-Friedrichs scheme up to arbitrary $T>0$ was derived by variational techniques.
 \item[(3)] Pointwise convergence of $u^k_m$ to $u=v_x$ was proved almost everywhere.
In particular, this yielded uniform convergence, except for neighborhoods of shocks with arbitrarily small measure.
 \item[(4)] Uniform convergence of $v^k_{m+1}$ to $v$ with an error $O(\sqrt{\dx})$ was proved from a stochastic viewpoint.
 \item[(5)] Random walks played a role as characteristic curves of the difference equations, which converged to the genuine characteristic curves of (\ref{CL2}) and (\ref{HJ2}).
\end{enumerate}

The purpose of this paper is to show further results for the Lax-Friedrichs scheme on the basis of (1)--(5) under the assumptions (A1)--(A4) with techniques from the theory of viscosity solutions of Hamilton-Jacobi equations.
We refer only to the results for the Lax-Friedrichs scheme, but similar results for other finite difference schemes with numerical viscosity (e.g., the upwind/downwind scheme) are available as well.
The main results are on the following:

1. Time-global stability of the Lax-Friedrichs scheme with a fixed mesh size.

2. Error estimates for entropy solutions.

It is proved that genuine entropy solutions at $t=1$ are uniformly bounded, regardless of the magnitudes of the initial data.
Since the genuine solutions are well approximated by the difference entropy solutions for small mesh sizes, the difference entropy solutions at $t=1$ are also uniformly bounded.
Due to the periodic setting, iteration of the time-1 analysis yields time-global properties.
Combining these facts, we obtain time-global stability of the Lax-Friedrichs scheme.
As a result, we can show that the large-time behavior of the Lax-Friedrichs scheme is such that any solutions associated with each $c$ fall into the time periodic state uniquely determined by each $c$.
This means that for each $c$ we obtain the unique space-time periodic difference entropy solution and the unique (up to a constant) space-time periodic difference viscosity solution.
These approximate the genuine $\Z^2$-periodic entropy (resp. viscosity) solution of (\ref{CL2}) (resp. (\ref{HJ2})).
For the periodic states, we naturally have the notion of the effective Hamiltonian for the difference Hamilton-Jacobi equation (\ref{HJ-Delta}). 
We reveal its properties and prove that it converges to the effective Hamiltonian for the exact equation (\ref{HJ2}) with an error estimate of $O(\sqrt{\dx})$. 

It is known that the optimal estimate of the $L^1$-error between $u^k_m$ and $u$ is $O(\sqrt{\dx})$ in the case of $H(x,t,p)=H(p)$ \cite{Sabac}.
The upper bound $O(\sqrt{\dx})$ is due to properties of functions of bounded variation \cite{Kuzentsov}.
It is not clear whether the result in~\cite{Kuzentsov} is applicable to the case of our general flux functions.
Through a different approach, we obtain an $L^1$-error estimate of $O(\dx^{\frac{1}{4}})$.
This error estimate is based on $O(\sqrt{\dx})$, which arises as the error between random walks and their space-time continuous limit under hyperbolic scaling (i.e., a backward characteristic curve).
For a technical reason, we lose the exponent $1/4$ in the case of the general flux function $H(x,t,p)$.
In addition, we show that if the genuine entropy solution is Lipschitz, then a $C^0$-error estimate of $O(\dx^{\frac{1}{4}})$ is available.

Unlike the case for initial value problems, it is challenging to show convergence of full sequences and  estimate the error for $\Z^2$-periodic entropy (resp. viscosity) solutions  of (\ref{CL2}) (resp. (\ref{HJ2})), because the uniqueness of such genuine $\Z^2$-periodic solutions with respect to $c$ is not valid in general.
However, we can manage the special case in which a genuine $\Z^2$-periodic entropy solution $\bu$ with some $c$ is $C^1$ and the dynamics of its characteristic curves $C^\ast(s):=(q(s)\mod 1,s\mod1)$  are $C^1$-conjugate to the dynamics of the linear flow on $\T^2$ with a Diophantine rotation vector.
Such a solution $\bu$ is known as a {\it KAM} torus in Hamiltonian dynamics (e.g., see \cite{JKM}, \cite{Moser}, \cite{Sevryuk}).
We show a $C^0$-error estimate depending on the Diophantine nature of the rotation vector, which is a rigorous result on finite difference approximation of KAM tori.
Our proof is based on the fact that one orbit of the linear flow on $\T^2$ with a Diophantine rotation vector is ergodic on $\T^2$ and hence so is each $C^\ast(s)$.

Finally, we note that our motivation comes not only from the viewpoint of PDEs in continuum mechanics but also from the recent theory of Lagrangian and Hamiltonian dynamics that is called the Aubry-Mather theory or the weak  KAM theory \cite{Fathi}, \cite{WE}, \cite{JKM}.
Our periodic setting is standard, and $\Z^2$-periodic entropy (resp. viscosity) solutions of  (\ref{CL2}) (resp. (\ref{HJ2})) and the effective Hamiltonian play central roles in the weak KAM theory.
The results of this paper provide basic tools for numerical analysis of the weak KAM theory through finite difference approximation.
We remark that from the standpoint of accuracy it is better to approximate entropy solutions and characteristic curves as well as viscosity solutions, because the central objects in the weak KAM theory, such as KAM tori, Aubry-Mather sets, effective Hamiltonians, and calibrated curves, are obtained from the derivatives of viscosity solutions or entropy solutions.
The ``derivatives'' of numerical viscosity solutions obtained through a scheme that has no relation to entropy solutions are not accurate in general.
Some developments in finite difference approximation methods and numerical simulations for the weak KAM theory are found in \cite{Nishida-Soga}.
However, the results are mathematically restricted by the absence of the stochastic and variational approach to the Lax-Friedrichs scheme.
We also point to \cite{Bessi} and \cite{JKM} for results on smooth approximation methods for the weak KAM theory based on the vanishing viscosity method. In particular, \cite{Bessi} successfully applies the stochastic and variational approach to the vanishing viscosity method given in \cite{Fleming}, where the genuine characteristic curves are approximated by solutions of stochastic ODEs with the standard Brownian motion.

The advantage of our stochastic and variational approach is that structures and properties similar to those of the exact equations (\ref{CL2}) and (\ref{HJ2}) are available in the most common finite difference schemes, which provides much more information on the schemes.
In particular, we can trace genuine characteristic curves by means of random walks.
This enables further development of finite difference approximation methods for the classical and weak KAM theories.
\setcounter{section}{1}
\setcounter{equation}{0}
\section{Preliminary Results}
In this section, we state several important preliminary results.
\subsection{Entropy Solution and Viscosity Solution}
It is well known that the viscosity solution $v$ of (\ref{HJ}) is Lipschitz and is characterized by the calculus of variations.
The value of $v$ at each point $(x,t)\in\T\times(0,T]$, $T\in(0,\infty)$, is given by
\begin{eqnarray}\label{value-func}
v(x,t)=\inf_{\gamma\in AC,\,\,\gamma(t)=x}\left\{ \int^t_0 L^{(c)}(\gamma(s),s,\gamma'(s))ds+v^0(\gamma(0)) \right\}+h(c)t,
\end{eqnarray}
where $AC$ is the family of absolutely continuous curves $\gamma:[0,t]\to\T$ and
$$L^{(c)}(x,t,\xi):=L(x,t,\xi)-c\xi$$
is the Legendre transform of $H(x,t,c+\cdot)$.
We can find a minimizing curve $\gamma^\ast$ of (\ref{value-func}) that is a backward characteristic curve of (\ref{CL2}) and (\ref{HJ2}) as well as a $C^2$-solution of the Euler-Lagrange equation generated by the Lagrangian $L^{(c)}$.
On each minimizing curve, $v$ is differentiable with respect to $x$:
\begin{eqnarray}\label{v_x}
v_x(\gamma^\ast(s),s)=L^{(c)}_\xi(\gamma^\ast(s),s,\gamma^\ast{}'(s)) \mbox{ for $0<s<t$.}
\end{eqnarray}
We say that a point $(x,t)$ is a regular point of $v$, or regular, if $v_x(x,t)$ exists.
Since $v$ is Lipschitz, almost every point is regular.
In particular, if $(x,t)$ is regular, the minimizing curve $\gamma^\ast$ for (\ref{value-func}) is unique and (\ref{v_x}) holds for $s=t$.

Usually, the entropy solution $u$ of (\ref{CL}) is defined as an element of $C^0((0,T];L^1(\T))$.
Here we always take the representative element given by $v_x$, which is still denoted by $u$.
If $(x,t)$ is regular and $\gamma^\ast$ is the unique minimizing curve for $v(x,t)$, the value of the entropy solution $u=v_x$ at the point $(x,t)$ is given by
\begin{eqnarray*}\label{value-func-CL}
u(x,t)=\int^t_0L^{(c)}_x(\gamma^\ast(s),s,\gamma^\ast{}'(s))ds+u^0(\gamma^\ast(0)),
\end{eqnarray*}
where $u^0$ is assumed to be rarefaction-free,
 $$\esssup\displaylimits_{x\neq y}\frac{u^0(x)-u^0(y)}{x-y}\le M\mbox{\,\,\, for some $M>0$ (i.e., a one-sided Lipschitz condition)},$$
or equivalently $v^0$ is semiconcave,
$$v^0(x+h)+v^0(x-h)-2v^0(x)\le Mh^2\mbox{ for all $x,h$.}$$
Otherwise, $u^0(\gamma^\ast(0))$ must be replaced with $L^{(c)}_\xi(\gamma^\ast(0),0,\gamma^\ast{}'(0))$.
In particular, for any $\tau\in[0,t)$ we have
$$u(x,t)=\int^t_\tau L^{(c)}_x(\gamma^\ast(s),s,\gamma^\ast{}'(s))ds+L^{(c)}_\xi(\gamma^\ast(\tau),\tau,\gamma^\ast{}'(\tau)).
$$
For more details see, e.g., \cite{Bernard} or \cite{Cannarsa}.

We introduce the solution operators of (\ref{CL}) and (\ref{HJ}) as follows:
$$\phi^t:L^\infty_{r,0}(\T)\ni u^0\mapsto u(\cdot,t)\in L^\infty(\T), \,\,\,\,\,\psi^t:Lip_{r}(\T)\ni v^0\mapsto v(\cdot,t)\in Lip(\T),$$
where $L^\infty_{r,0}(\T)$ is the set of all functions $u^0\in L^\infty(\T)$  with $\norm u^0\norm_{L^\infty}\le r$ and $\int _\T u^0dx=0$, while $Lip_r(\T)$ is the set of all Lipschitz functions on $\T$ with a Lipschitz constant bounded by $r$.
When we specify the value of $c$, we write $\phi^t(\cdot;c)$, $\psi^t_\Delta(\cdot;c)$, $u^{(c)}$, $v^{(c)}$.

We would like to prove a priori boundedness of $u(x,t)=v_x(x,t)$.
This is closely related to a priori compactness of minimizers for (\ref{value-func}).
We remark that a priori compactness of minimizers plays an important role in the Aubry-Mather theory and the weak KAM theory, and details are known for more general settings (e.g., \cite{Mather}, \cite{Iturriaga}).
The basic assumptions for this are (A1)$'$--(A3)$'$ and completeness of the Euler-Lagrange flow.
Here we adopt (A4), which is stronger than the completeness assumption.
We need this to obtain compactness of minimizers for our stochastic and variational problems, which do not satisfy the Euler-Lagrange equation generated by $L^{(c)}$.
In order to provide a self-contained treatment, we give brief proofs by modifying Section~4.1 of~\cite{Fathi-book}.
\begin{Prop} \label{boundedness}
For each $t\in(0,T]$, there exists a constant $\beta_1(t)>0$ (independent of $r$, $c\in[c_0,c_1]$, and the initial data $v^0, u^0$) for which
$$\norm \phi^t(u^0;c)\norm_{L^\infty}\le\beta_1(t),\quad \norm \psi^t(v^0;c)_x\norm_{L^\infty}\le\beta_1(t).$$
\end{Prop}
\begin{proof}[{\bf Proof.}]
Fix $t\in(0,T]$.
If $(x,t)$ is regular, then (\ref{v_x}) holds for $s=t$.
Thus, it is sufficient to estimate $L^{(c)}_\xi(\gamma^\ast(t),t,\gamma^\ast{}'(t))$ for each minimizing curve $\gamma^\ast$ of (\ref{value-func}).
We now prepare two lemmas.
\begin{Lemma} \label{1}
Let $\gamma^\ast$ be a minimizing curve for $v(x,t)$.
Set $y:=\gamma^\ast(0)$.
Then, $\gamma^\ast$ attains
$$\inf_{\gamma\in AC,\gamma(t)=x,\gamma(0)=y}\int^t_0 L^{(c)}(\gamma(s),s,\gamma'(s))ds.$$
\end{Lemma}
\begin{proof}[{\bf Proof.}]
If not, there exists $\gamma^\sharp$ such that
$$\int^t_0 L^{(c)}(\gamma^\sharp(s),s,\gamma^\sharp{}'(s))ds<\int^t_0 L^{(c)}(\gamma^\ast(s),s,\gamma^\ast{}'(s))ds.$$
Since $v^0(\gamma^\sharp(0))=y=v^0(\gamma^\ast(0))$, we have
$$\int^t_0 L^{(c)}(\gamma^\sharp(s),s,\gamma^\sharp{}'(s))ds+v^0(\gamma^\sharp(0))<\int^t_0 L^{(c)}(\gamma^\ast(s),s,\gamma^\ast{}'(s))ds+v^0(\gamma^\ast(0)).$$
Therefore, $\gamma^\ast$ is not a minimizing curve for $v(x,t)$, which is a contradiction.
\end{proof}
We define the following set:
$$\Gamma(t):=\left\{ \gamma^{(c)} \,\,|\,\,\gamma^{(c)}\mbox{ attains } \inf_{\gamma(t)=x,\gamma(0)=y}\int^t_0 L^{(c)}(\gamma(s),s,\gamma'(s))ds,\,\,\,x,y\in\T,c\in[c_0,c_1]   \right\}.$$
By Lemma ~\ref{1}, any minimizing curve $\gamma^\ast$ for $v(x,t)$, $x\in\T$, belongs to $\Gamma(t)$.
(Actually, we should take $\gamma^\ast\mod1$, but this is not important due to the periodic setting.)
\begin{Lemma} \label{2}
\begin{enumerate}
\item There exists a constant $C_1(t)>0$ such that for any $x,y\in\T$ we have a $C^1$-curve $\gamma$ that satisfies
$$\gamma(t)=x,\,\,\,\gamma(0)=y,\,\,\,\int_0^tL^{(c)}(\gamma(s),s,\gamma'(s))ds\le C_1(t).$$
In particular,  any $\gamma^{(c)}\in\Gamma(t)$ satisfies
$$\int_0^tL^{(c)}(\gamma^{(c)}(s),s,\gamma^{(c)}{}'(s))ds\le C_1(t).$$
\item There exists a constant $C_2(t)>0$ such that for any $\gamma^{(c)}\in\Gamma(t)$ we have $\tau\in(0,t)$ that satisfies
$$|\gamma^{(c)}{}'(\tau)|\le C_2(t).$$
\item There exists a constant $C_3(t)>0$ such that for any $\gamma^{(c)}\in\Gamma(t)$ we have
$$|L^{(c)}_\xi(\gamma^{(c)}(s),s,\gamma^{(c)}{}'(s))|\le C_3(t),\quad s\in[0,t].$$
\end{enumerate}
\end{Lemma}
\begin{proof}[{\bf Proof.}]
1. Consider $\gamma(s):=x+\frac{x-y}{t}(s-t)$.
Since $|x-y|\le 1$, we have $|\gamma'(s)|\le t^{-1}$.
Therefore, we obtain
$$\int_0^tL^{(c)}(\gamma(s),s,\gamma'(s))ds\le \sup_{x,s\in\T,|\xi|\le t^{-1},c\in[c_0,c_1]}|L^{(c)}(x,s,\xi)|t.$$
Set $C_1(t):=\sup_{x,s\in\T,|\xi|\le t^{-1},c\in[c_0,c_1]}|L^{(c)}(x,s,\xi)|t$ and Claim~1 is proved.

2. Due to Claim~1 and the minimizing property of $\gamma^{(c)}$, we have $\tau\in(0,t)$ that satisfies
$$C_1(t)\ge\int_0^tL^{(c)}(\gamma^{(c)}(s),s,\gamma^{(c)}{}'(s))ds=L^{(c)}(\gamma^{(c)}(\tau),\tau,\gamma^{(c)}{}'(\tau))t.$$
By (A3), $|\gamma^{(c)}{}'(\tau)|$ must be bounded by a constant $C_2(t)$ independent of $\gamma^{(c)}\in\Gamma(t)$.

3.  Note that $\gamma^{(c)}$ is a $C^2$-solution of the following Euler-Lagrange equation generated by $L^{(c)}$:
$$\frac{d}{dt}L^{(c)}_\xi(\gamma^{(c)}(s),s,\gamma^{(c)}{}'(s))=L_x(\gamma^{(c)}(s),s,\gamma^{(c)}{}'(s)).$$
It follows from (A1)--(A4) that there exists $\alpha_1$ for which $|L^{(c)}_x|\le\alpha_1(|L^{(c)}|+1)$ for any $c\in[c_0,c_1]$ and that $\dis L_\ast:=|\min\{0,\inf_{x,s,\xi,c}L^{(c)}\}|$ is bounded.
We have $\tau^\ast\in[0,t]$, which attains the maximum of $|L^{(c)}_\xi(\gamma^{(c)}(s),s,\gamma^{(c)}{}'(s))|$, $0\le s\le t$. Suppose that  $\tau^\ast\neq\tau$, where $\tau$ is the value in Claim~2.
Then,
\begin{eqnarray*}
|  \int^{\tau^\ast}_\tau\frac{d}{dt}L^{(c)}_\xi(\gamma^{(c)}(s),s,\gamma^{(c)}{}'(s))ds |
&=&|L^{(c)}_\xi(\gamma^{(c)}(\tau^\ast),\tau^\ast,\gamma^{(c)}{}'(\tau^\ast))-L^{(c)}_\xi(\gamma^{(c)}(\tau),\tau,\gamma^{(c)}{}'(\tau))|\\
&\le&\int_0^t|L^{(c)}_x(\gamma^{(c)}(s),s,\gamma^{(c)}{}'(s))|ds\\
&\le&\int_0^t\alpha_1(1+|L^{(c)}(\gamma^{(c)}(s),s,\gamma^{(c)}{}'(s))|)ds \\
&\le&\alpha_1\int_0^t1+(L^{(c)}(\gamma^{(c)}(s),s,\gamma^{(c)}{}'(s))+L_\ast)+L_\ast ds\\
&=&\alpha_1(2L_\ast+1)t+\alpha_1\int_0^tL^{(c)}(\gamma^{(c)}(s),s,\gamma^{(c)}{}'(s))ds\\
&\le&\alpha_1(2L_\ast+1)t+\alpha_1C_1(t).
\end{eqnarray*}
Therefore, setting
$$C_3(t):=\alpha_1(2L_\ast+1)t+\alpha_1C_1(t)+\sup_{x,s\in\T,|\xi|\le C_2(t),c\in[c_0,c_1]}|L^{(c)}_\xi(x,s,\xi)|,$$
for $0\le s\le t$ we obtain
$$|L^{(c)}_\xi(\gamma^{(c)}(s),s,\gamma^{(c)}{}'(s))|\le|L^{(c)}_\xi(\gamma^{(c)}(\tau^\ast),\tau^\ast,\gamma^{(c)}{}'(\tau^\ast))|\le C_3(t).$$
The case $\tau^\ast=\tau$ is included by the above inequality.
\end{proof}

Since $(x,t)$ is regular for almost every $x\in\T$ with each fixed $t$ and $v_x(x,t)=L^{(c)}_\xi(\gamma^\ast(t),t,\gamma^\ast{}'(t))$ holds for almost every $x\in\T$, we obtain Proposition~\ref{boundedness} by setting $\beta_1(t):=C_3(t)$.
\end{proof}
We show continuity of $\phi^t(v^0_x;c)$ and  $\psi^t(v_0;c)$ with respect to $v^0$ and $c$.
\begin{Prop}\label{continuity}
Fix $t\in(0,T]$.
For each sequence $v^0_j\to v^0$ uniformly and  $c^j\to c$ as $j\to\infty$ ($v^0_j{}_x$ is not necessarily convergent), we have
\begin{eqnarray*}
\psi^t(v^0_j;c^j)\to\psi^t(v^0;c)\mbox{ uniformly,\,\,\, }\phi^t(v^0_j{}_x;c^j)\to\phi^t(v^0_x;c)\mbox{ in $L^1(\T)$ \,\,\,\,\,\,as $j\to\infty$.}
\end{eqnarray*}
\end{Prop}
\begin{proof}[{\bf Proof.}]
By the variational representation, we have
\begin{eqnarray*}
\psi^t(v^0;c)(x)&=& \int^t_0 L(\gamma^\ast(s),s,\gamma^\ast{}'(s))-c\gamma^\ast{}'(s)ds+v^0(\gamma^\ast(0))+h(c)t,\\
\psi^t(v^0_j;c^j)(x)&=& \int^t_0 L(\gamma^\ast_j(s),s,\gamma^\ast_j{}'(s))-c^j\gamma^\ast_j{}'(s)ds+v^0_j(\gamma^\ast_j(0))+h(c^j)t
\end{eqnarray*}
and hence
\begin{eqnarray*}
\psi^t(v^0_j;c^j)(x)-\psi^t(v^0;c)(x)&\le& \int^t_0 -(c^j-c)\gamma^\ast{}'(s)ds+v^0_j(\gamma^\ast(0))-v^0(\gamma^\ast(0))\\
&& +(h(c^j )-h(c ))t,\\
\psi^t(v^0_j;c^j)(x)-\psi^t(v^0;c)(x)&\ge&  \int^t_0 -(c^j-c)\gamma^\ast_j{}'(s)ds+v^0_j(\gamma^\ast_j(0))-v^0(\gamma^\ast_j(0))\\
&& +(h(c^j )-h(c ))t.
\end{eqnarray*}
It follows from Claim~3 of Lemma~\ref{2} that  any minimizing curves for $v(x,t)$  are Lipschitz with a common Lipschitz constant for all  $x\in\T$ and  $v^0\in Lip_r(\T)$.
Since $h$ is continuous, we conclude that $\psi^t(v^0_j;c^j)\to\psi^t(v^0;c)$ uniformly as $j\to\infty$.

Let $x\in\T$ be a common regular point of all $\psi^t(v^0_j;c^j)$, $j= 1,2,3,\ldots$.
Almost every point is such a point.
Through a variational technique, we find that $\gamma^\ast_j\to\gamma^\ast$ uniformly and $\gamma^\ast_j{}'\to\gamma^\ast{}'$ in $L^2$ as $j\to\infty$  (e.g., see Lemma~3.4 in~\cite{Soga2}).
Note that for each $0\le \tau<t$ we have
\begin{eqnarray*}
&&\phi^t(v^0_x;c)(x)=\psi^t(v^0;c)_x(x)=\int^t_\tau L_x(\gamma^\ast(s),s,\gamma^\ast{}'(s))ds+L_\xi(\gamma^\ast(\tau),\tau,\gamma^\ast{}'(\tau))-c,\\
&&\phi^t(v^0_j{}_x;c^j)(x)=\psi^t(v^0_j;c^j)_x(x)=\int^t_\tau L_x(\gamma^\ast_j(s),s,\gamma^\ast_j{}'(s))ds+L_\xi(\gamma^\ast_j(\tau),\tau,\gamma^\ast_j{}'(\tau))-c^j.
\end{eqnarray*}
For any $\varepsilon>0$, there exists a $J$ such that, if $j\ge J$, we have $\norm \gamma^\ast_j-\gamma^\ast\norm_{C^0}\le\varepsilon$ and $\norm \gamma^\ast_j{}'-\gamma^\ast{}'\norm_{L^2}\le\varepsilon\sqrt{t}$.
Note that we have $\tau$ (depending on $j\ge J$) such that $|\gamma^\ast_j{}'(\tau)-\gamma^\ast{}'(\tau)|\le\varepsilon$.
Therefore, we conclude that  $\phi^t(v^0_j{}_x;c^j)\to\phi^t(v^0_x;c)$ pointwise almost everywhere.
This immediately leads to $L^1(\T)$-convergence.
\end{proof}
\subsection{Stochastic and Variational Approach to the Lax-Friedrichs Scheme}
In this subsection, we state several results of the stochastic and variational approach to the Lax-Friedrichs scheme that are shown in~\cite{Soga2}.
Let $N,K$ be natural numbers with $N\le K$.
The mesh size $\Delta=(\dx,\dt)$ is defined by $\dx:=(2N)^{-1}$ and $\dt:=(2K)^{-1}$.
We set $\lambda:=\dt/\dx$.
We also set $x_m:=m\dx$ for $m\in\Z$ and $t_k:=k\dt$ for $k=0,1,2,\ldots$.
For $x\in\R$ and $t>0$, the notation $m(x),k(t)$ denotes the integers $m,k$ for which $x\in[x_m,x_m+2\dx)$ and $t\in[t_k,t_k+\dt)$, respectively.
Let $(\dx\Z)\times(\dt\Z_{\ge0})$ be the set of all $(x_m,t_k)$, and let
$$\mathcal{G}_{even}\subset (\dx\Z)\times(\dt\Z_{\ge0}),\qquad\mathcal{G}_{odd}\subset (\dx\Z)\times(\dt\Z_{\ge0})$$
be the set of all $(x_m,t_k)$ with $k=0,1,2,\ldots$ and $m\in\Z$ such that  $m+k$ is even (odd), which is called the even grid (odd grid).
We consider the discretization of (\ref{CL}) by the Lax-Friedrichs scheme in $\mathcal{G}_{even}$:
\begin{eqnarray}\label{2CL-Delta}
\left\{
\begin{array}{lll}
&\dis \frac{u^{k+1}_{m+1}-\frac{(u^k_{m}+u^k_{m+2})}{2}}{\dt}
+\frac{H(x_{m+2},t_k,c+u^k_{m+2})-H(x_m,t_k,c+u^k_m)}{2\dx}
=0, \\\\\medskip\medskip
&u^0_m=u^0_\Delta(x_m),\quad u^k_{m\pm 2N}=u^k_{m},
\end{array}
\right.
\end{eqnarray}
where for $m$ even
\begin{eqnarray}\label{u^0}
u_\Delta^0(x):=\frac{1}{2\dx}\int^{x_{m}+\dx}_ {x_m-\dx}u^0(y)dy\mbox{\quad for $x\in[x_m-\dx,x_{m}+\dx)$}.
\end{eqnarray}
Note that $\dis\sum_{\{m\,|\,0\le m< 2N,\,m+k\;\mbox{even}\}}u^k_m\cdot2\dx$ is conservative with respect to $k$ and is zero for $u^0$ that  has zero mean.
We also discretize (\ref{HJ}) in $\mathcal{G}_{odd}$:
\begin{eqnarray}\label{2HJ-Delta}\quad
\left\{
\begin{array}{lll}
&\dis \frac{ v^{k+1}_{m} - \frac{(v^k_{m-1}+v^k_{m+1})}{2} }{\dt}
+H(x_{m},t_k,c+\frac{v^k_{m+1}-v^k_{m-1}}{2\dx})
=h(c),\\\\\medskip\medskip
&v^0_{m+1}=v^0_\Delta(x_{m+1}),\quad v^k_{m+1\pm 2N}=v^k_{m+1},
\end{array}
\right.
\end{eqnarray}
where, in addition to $u^0=v^0_x$, we assume that
\begin{eqnarray}\label{v^0}
\,\,\,\,\,\,v^0_\Delta(x):=v^0(-\dx)+\int^x_{-\dx}u^0_\Delta(y)dy\mbox{ (i.e., $v^0_\Delta(x_{m+1})=v^0(x_{m+1})$ for $m$ even).}
\end{eqnarray}
Note that $u_\Delta^0\to u^0$ in $L^1(\T)$ and $v_\Delta^0\to v^0$ uniformly with $\norm v^0_\Delta-v^0\norm_{C^0}\le \norm u^0\norm_{L^\infty}\cdot 2\dx$, as $\Delta\to0$.
We introduce the following difference operators:
$$D_tw^{k+1}_m:=\frac{w^{k+1}_m-\frac{w^k_{m-1}+w^k_{m+1}}{2}}{\dt},\quad D_xw^{k}_{m+1}:=\frac{w^{k}_{m+1}-w^k_{m-1}}{2\dx}.$$
The two problems (\ref{2CL-Delta}) and (\ref{2HJ-Delta}) are equivalent under (\ref{u^0}) and (\ref{v^0}).
In particular, we have $D_xv^k_{m+1}=u^k_m$ \cite{Soga2}.
Let $u_\Delta$ be the step function derived from the solution $u^k_m$ of (\ref{2CL-Delta}); namely,
$$\mbox{$u_\Delta(x,t):=u^k_m$ for $(x,t)\in[x_{m-1},x_{m+1})\times[t_k,t_{k+1})$.}$$
Let $v_\Delta$ be the linear interpolation with respect to the space variable derived from the solution $v^k_{m+1}$ of (\ref{2HJ-Delta}); namely,
$$\mbox{$v_\Delta(x,t):=v^k_{m-1}+D_xv^k_{m+1}\cdot(x-x_{m-1})$ for $(x,t)\in[x_{m-1},x_{m+1})\times[t_k,t_{k+1})$.}$$
We remark that $v_\Delta(x,\cdot)$ is a step function for each fixed $x$ and that $(v_\Delta)_x=u_\Delta$.

We introduce space-time inhomogeneous random walks in $\mathcal{G}_{odd}$, which correspond to characteristic curves of (\ref{CL}) and (\ref{HJ}).
For each point  $(x_n,t_{l+1})\in\mathcal{G}_{odd}$, we introduce backward random walks $\gamma$ that start from $x_n$ at $t_{l+1}$ and move by $\pm\dx$ in each backward time step:
$$\gamma=\{\gamma^k\}_{k=0,1,\ldots,l+1},\quad\gamma^{l+1}=x_{n},\quad \gamma^{k+1}-\gamma^k=\pm\dx.$$
More precisely,  for each $(x_n,t_{l+1})\in\mathcal{G}_{odd}$ we introduce the following:
\begin{eqnarray*}
&&X^k:=\{ x_m \,|\, \mbox{ $(x_m,t_k)\in\mathcal{G}_{odd}$, $|x_m-x_n|\le(l+1-k)\dx$}\}\mbox{ for }k\le l+1,\\
&&G:=\bigcup_{1\le k\le l+1}\big(X^{k}\times\{t_{k}\}\big)\subset\mathcal{G}_{odd}, \\
&&\xi:G\ni(x_m,t_k)\mapsto\xi^k_m\in[-\lambda^{-1},\lambda^{-1}],\quad \lambda=\dt/\dx, \\
&&\bar{\bar{\rho}}: G\ni(x_m,t_k)\mapsto\bar{\bar{\rho}}^k_m:=\frac{1}{2}-\frac{1}{2}\lambda\xi^k_m\in[0,1],\\
&&\bar{\rho}: G\ni(x_m,t_k)\mapsto\bar{\rho}^k_m:=\frac{1}{2}+\frac{1}{2}\lambda\xi^k_m\in[0,1],\\
&&\gamma:\{ 0,1,2,\ldots,l+1\}\ni k\mapsto \gamma^k\in X^k,\mbox{ $\gamma^{l+1}=x_n$, $\gamma^{k+1}-\gamma^k=\pm\dx$},\\
&&\Omega:\mbox{the family of the above $\gamma$}.
\end{eqnarray*}
We regard $\bar{\bar{\rho}}^k_m$ (resp. $\bar{\rho}^k_m$) as the probability of transition from $(x_m,t_k)$ to $(x_m+\dx,t_k-\dt)$ (resp. from $(x_m,t_k)$  to $(x_m-\dx,t_k-\dt)$).
Note that $\xi$ is a control for random walks, which plays the role of a velocity field on the grid.
We define the density of each path $\gamma\in\Omega$ as
$$\mu(\gamma):=\prod_{1\le k\le l+1}\rho(\gamma^{k},\gamma^{k-1}),$$
where $\rho(\gamma^{k},\gamma^{k-1})=\bar{\bar{\rho}}^k_{m(\gamma^{k})}$ (resp. $\bar{\rho}^k_{m(\gamma^{k})}$) if $\gamma^{k}-\gamma^{k-1}=-\dx$ (resp. $\dx$).
The density $\mu(\cdot)=\mu(\cdot;\xi)$ yields a probability measure for $\Omega$; namely,
$$prob(A)=\sum_{\gamma\in A}\mu(\gamma;\xi)\mbox{\quad for $A\subset\Omega$}. $$
The expectation with respect to this probability measure is denoted by $E_{\mu(\cdot;\xi)}$; namely, for a random variable $f:\Omega\to\R$ we have
$$E_{\mu(\cdot;\xi)}[f(\gamma)]:=\sum_{\gamma\in\Omega}\mu(\gamma;\xi)f(\gamma).$$
We use $\gamma$ as the symbol for random walks or a sample path.
If necessary, we write $\gamma=\gamma(x_n,t_{l+1};\xi)$ in order to specify its initial point and control.

We now state an important result on the scaling limit of inhomogeneous random walks.
Let $\eta(\gamma)=\{\eta^k(\gamma)\}_{k=0,1,2,\ldots,l+1}$, $\gamma\in\Omega$ be a random variable that is  induced by a random walk $\gamma=\gamma(x_n,t_{l+1};\xi)$ and is defined by
$$\eta^{l+1}:=\gamma^{l+1},\,\,\,\,\,\, \eta^k(\gamma):=\gamma^{l+1}-\sum_{k< k'\le l+1}\xi(\gamma^{k'},t_{k'})\dt\mbox{ \,\,\, for $0\le k\le l$}.$$
\begin{Prop}(\cite{Soga1})\label{eta}
Set $\dis \tilde{\sigma}^k:=E_{\mu(\cdot;\xi)}[|\gamma^k-\eta^k(\gamma)|^2]$ and $\dis \tilde{d}^k:=E_{\mu(\cdot;\xi)}[|\gamma^k-\eta^k(\gamma)|]$ for $0\le k\le l+1$.
Then, we have
$$(\tilde{d}^k)^2\le\tilde{\sigma}^k\le \frac{t^{l+1}-t^k}{\lambda}\dx.$$
\end{Prop}
\noindent If we take the hyperbolic scaling limit, in which $\Delta=(\dx,\dt)\to0$ under
$$0<\lambda_0\le\lambda=\dt/\dx<\lambda_1,$$
then $\tilde{d}^k$ and $\sqrt{\tilde{\sigma}^k}$ always tend to zero with $O(\sqrt{\dx})$.
Note that the variance does not necessarily do so for inhomogeneous random walks.
We refer to~\cite{Soga1} for more details of the hyperbolic scaling limit of inhomogeneous random walks.
Note that we always take the limit $\Delta\to0$ under hyperbolic scaling.

Now we state results for the stochastic and variational approach to the Lax-Friedrichs scheme.
\begin{Thm}[\cite{Soga2}]\label{SVA}
There exists $\lambda_1>0$ (depending  on $T$, $[c_0,c_1]$, and $r$) such that for any small $\Delta=(\dx,\dt)$ with $\lambda=\dt/\dx<\lambda_1$ we have the following:
\begin{enumerate}
\item The expectation
\begin{eqnarray*}\label{action-delta}
E_{\mu(\cdot;\xi)}\Big[\sum_{0<k\le l+1}L^{(c)}(\gamma^k,t_{k-1},\xi^k_{m(\gamma^k)})\dt +v^0_\Delta(\gamma^0)\Big]+h(c)t_{l+1}
\end{eqnarray*}
which is given by $\gamma=\gamma(x_n,t_{l+1};\xi)$, has an infimum with respect to $\xi:G\to[-\lambda^{-1},\lambda^{-1}]$ for each $n\in\Z$ and $0<l+1\le k(T)$.
The infimum is attained by the $\xi^\ast $ that satisfies $|\xi^\ast|\le\lambda_1^{-1}<\lambda^{-1}$.
\item  For each $n\in\Z$ and $0< l+1 \le k(T)$ the solution of (\ref{2HJ-Delta}) satisfies
$$v^{l+1}_n=\inf_{\xi}E_{\mu(\cdot;\xi)}\Big[\sum_{0<k\le l+1}L^{(c)}(\gamma^k,t_{k-1},\xi^k_{m(\gamma^k)})\dt +v^0_\Delta(\gamma^0)\Big]+h(c)t_{l+1}.$$
\item  For each $v^{l+1}_n$ the minimizing velocity field $\xi^\ast$ is unique and in $G$ satisfies
$$L^{(c)}_\xi(x_m,t_{k},\xi^\ast{}^{k+1}_m)=D_xv^{k}_{m+1}\,\,\,(\Leftrightarrow
\xi^\ast{}^{k+1}_m=H_p(x_m,t_{k},c+D_xv^{k}_{m+1})).$$
\item Let $\xi^\ast$ (resp. $\tilde{\xi}^\ast$) be the minimizing velocity field for $v^{l+1}_n$ (resp. $v^{l+1}_{n+2}$). Let $\gamma=\gamma(x_n,t_{l+1};\xi^\ast)$ and $\mu(\cdot;\xi^\ast)$ (resp. $\tilde{\gamma}=\gamma(x_{n+2},t_{l+1};\tilde{\xi}^\ast)$ and  $\tilde{\mu}(\cdot;\tilde{\xi}^\ast)$) be the minimizing random walk and its probability measure generated by $\xi^\ast$ (resp. $\tilde{\xi}^\ast$).
Then, $u^{l+1}_{n+1}=D_xv^{l+1}_{n+2}$ satisfies
\begin{eqnarray*}\label{5value-entropy}
u^{l+1}_{n+1}&\le& E_{\mu(\cdot;\xi^\ast)}\Big[\sum_{0<k\le l+1}L^{(c)}_x(\gamma^k,t_{k-1},\xi^\ast{}^k_{m(\gamma^k)})\dt+u^0_\Delta(\gamma^0+\dx) \Big]+O(\dx),
\\ \label{5value-entropy2}
u^{l+1}_{n+1}&\ge& E_{\tilde{\mu}(\cdot;\tilde{\xi}^\ast)}\Big[\sum_{0<k\le l+1}L^{(c)}_x(\tilde{\gamma}^k,t_{k-1},\tilde{\xi}^\ast{}^k_{m(\tilde{\gamma}^k)})\dt
+u^0_\Delta(\tilde{\gamma}^0-\dx) \Big]+O(\dx),
\end{eqnarray*}
where $O(\dx)$ stands for a number of $(-\theta\dx,\theta\dx)$ with $\theta>0$ independent of $\dx$.
\end{enumerate}
Now we take the hyperbolic scaling limit.
\begin{enumerate}
\item[5.] Let $v$ be the viscosity solution of (\ref{HJ}).
Then, for each $t\in[0,T]$ we have
$$\mbox{$v_\Delta(\cdot,t)\to v(\cdot,t)$ uniformly on $\T$ as $\Delta\to 0$.}$$
In particular, we have an error estimate.
That is, there exists $\beta_2>0$ (independent of $\Delta$, $c\in[c_0,c_1]$, and the initial data $v^0\in Lip_r(\T)$) such that
$$\sup_{t\in[0,T]}\norm v_\Delta(\cdot,t)-v(\cdot,t)\norm_{C^0(\T)}\le \beta_2\sqrt{\dx}.$$
\item[6.] Let $(x,t)\in\T\times(0,T]$ be a regular point and let $\gamma^\ast:[0,t]\to\R$ be the minimizing curve for $v(x,t)$.
Let $(x_n,t_{l+1})$ be a point of $[x-2\dx,x+2\dx)\times[t-\dt,t+\dt)$ and let $\gamma_\Delta:[0,t]\to\R$ be the linear interpolation of the random walk $\gamma=\gamma(x_n,t_{l+1};\xi^\ast)$ given by the minimizing velocity field $\xi^\ast$ for $v^{l+1}_n$.
Then,
$$\gamma_\Delta\to\gamma^\ast\mbox{ uniformly on $[0,t]$ in probability as  $\Delta\to0$}.$$
In particular, the average of $\gamma_\Delta$ converges uniformly to $\gamma^\ast$ as $\Delta\to0$.
\item[7.] Let $u=v_x$ be the entropy solution of (\ref{CL}).
Then, for each regular point $(x,t)\in\T\times[0,T]$ we have
$$u_\Delta(x,t)\to u(x,t)\mbox{ as $\Delta\to0$}.$$
In particular, $u_\Delta$ converges uniformly to $u$ on $(\T\times[0,T])\setminus\Theta$, where $\Theta$ is a neighborhood of the set of points of discontinuity of $u$ with an arbitrarily small measure.
\end{enumerate}
\end{Thm}
Note that Claims~1 and~3 give the stability condition of the Lax-Friedrichs scheme,
$$|\lambda H_p(x_m,t_{k},c+u^{k}_{m}))|<1,$$
which is called the {\it CFL~condition}.
\indent We next state further preliminary results for the Lax-Friedrichs scheme.
The solution operators of (\ref{2CL-Delta}) and (\ref{2HJ-Delta}) are introduced as
$$\phi^t_\Delta:L^\infty_{r,0}(\T)\ni u^0\mapsto u_\Delta(\cdot,t)\in L^\infty(\T),\,\,\,\,
\psi^t_\Delta:Lip_r(\T)\ni v^0\mapsto v_\Delta(\cdot,t)\in Lip(\T).$$
When we specify the value of $c$, we write $\phi^t_\Delta(\cdot;c), \psi^t_\Delta(\cdot;c),u_\Delta^{(c)}, u^k_m(c ),v_\Delta^{(c)}, v^k_{m+1}(c )$.
Note that we first obtain the step function $u^0_\Delta$ from $u^0$ with (\ref{u^0}) and then we map $u^0_\Delta$ to $u_\Delta(\cdot,t)$ with $\phi^t_{\Delta}$.
Similarly, we first obtain the piecewise linear function $v^0_\Delta$ from $v^0$ with (\ref{v^0}), in which $u^0=v^0_x$, and then we map $v^0_\Delta$ to $v_\Delta(\cdot,t)$ with $\psi^t_\Delta$.
\begin{Prop}\label{continuity-Delta}
Fix $t\in[0,T]$.
For each sequence $v^0_j\to v^0$ uniformly and  $c^j\to c$ as $j\to\infty$ ($v^0_j{}_x$ is not necessarily convergent), we have
\begin{eqnarray*}
\psi^t_\Delta(v^0_j;c^j)\to\psi^t_\Delta(v^0;c)\mbox{ uniformly,\,\,\, }\phi^t_\Delta(v^0_j{}_x;c^j)\to\phi^t_\Delta(v^0_x;c)\mbox{ in $L^1(\T)$ \,\,\,\,\,\,as $j\to\infty$.}
\end{eqnarray*}
\end{Prop}
\begin{proof}[{\bf Proof.}]
It is sufficient to show that $\psi^{t_{l+1}}_\Delta(v^0_j;c^j)(x_n)\to\psi^{t_{l+1}}_\Delta(v^0;c)(x_n)$ uniformly with respect to $x_n$ as $j\to\infty$.
Using the stochastic and variational representation, we have
\begin{eqnarray*}
\psi^{t_{l+1}}_\Delta(v^0;c)(x_n)&=&E_{\mu(\cdot;\xi^\ast)}\Big[\sum_{0<k\le l+1}L(\gamma^k,t_{k-1},\xi^\ast{}^k_{m(\gamma^k)})-c\xi^\ast{}^k_{m(\gamma^k)}\dt +v^0_\Delta(\gamma^0)\Big]\\
&&+h(c)t_{l+1},\\
\psi^{t_{l+1}}_\Delta(v^0_j;c^j)(x_n)&=&E_{\mu(\cdot;\xi_j^\ast)}\Big[\sum_{0<k\le l+1}L(\gamma^k,t_{k-1},\xi_j^\ast{}^k_{m(\gamma^k)})-c^j\xi_j^\ast{}^k_{m(\gamma^k)}\dt +v^0_{j\Delta}(\gamma^0)\Big]\\
&&+h(c^j)t_{l+1},
\end{eqnarray*}
where $\xi^\ast,\xi^\ast_j$ are minimizing velocity fields.
Hence, by the stochastic and variational representation again, we have
\begin{eqnarray*}
\psi^{t_{l+1}}_\Delta(v^0_j;c^j)(x_n)&-&\psi^{t_{l+1}}_\Delta(v^0;c)(x_n)\\
&\le&E_{\mu(\cdot;\xi^\ast)}\Big[\sum_{0<k\le l+1}-(c^j-c)\xi^\ast{}^k_{m(\gamma^k)}\dt +v^0_{j\Delta}(\gamma^0)-v^0_\Delta(\gamma^0)\Big]\\
&&+(h(c^j)-h(c))t_{l+1},\\
\psi^{t_{l+1}}_\Delta(v^0_j;c^j)(x_n)&-&\psi^{t_{l+1}}_\Delta(v^0;c)(x_n)\\
&\ge&E_{\mu(\cdot;\xi_j^\ast)}\Big[\sum_{0<k\le l+1}-(c^j-c)\xi_j^\ast{}^k_{m(\gamma^k)}\dt +v^0_{j\Delta}(\gamma^0)-v^0_\Delta(\gamma^0)\Big]\\
&&+(h(c^j)-h(c))t_{l+1}.
\end{eqnarray*}
Since $\xi^\ast,\xi^\ast_j$ are uniformly bounded, we have demonstrated the assertion.

The second convergence follows from the first one and the following relation:
$$\phi^t_\Delta(v^0_j{}_x;c^j)(x_m)= \frac{\psi^t_\Delta(v^0_j;c^j)(x_{m+1})-\psi^t_\Delta(v^0_j;c^j)(x_{m-1})}{2\dx}.$$
\end{proof}
\indent We now show details of the one-sided Lipschitz condition on $u^k_m$, or equivalently the semiconcave property of $v^k_{m+1}$; namely, we obtain the $\Delta$-independent upper boundedness of
$$E^k_\Delta:=\max_{m}\frac{u^k_{m+2}-u^k_m}{2\dx}=\max_m\frac{v^k_{m+3}+v^k_{m-1}-2v^k_{m+1}}{(2\dx)^2}.$$
This leads to the entropy condition on $u(\cdot,t)$ and semiconcavity of $v(\cdot,t)$.
The one-sided Lipschitz condition on $u^k_m$ is essential in the standard $L^1$-framework of difference approximation, because this condition yields $\Delta$-independent boundedness of the total variation of $u_\Delta(\cdot,t)$ and then $L^1$-convergence of the approximation follows with the aid of the compactness of functions of bounded variation.
{\it We remark that Theorem~\ref{SVA} was proved independently of the condition and without such compactness.}
In the sections below, we use the one-sided Lipschitz condition on $u^k_m$ for different purposes.

If we assume that $v^0$ is semiconcave, it is easy to find an upper bound for $E^k_\Delta$ through the semiconcavity of $v^{k}_{m+1}$ due to its variational structure.
However, we would like to avoid that assumption and know about the $k$-dependence of the upper bound.
Therefore, we use a direct method similar to that of Lemma~2 in~\cite{Oleinik}.
The direct method is available for arbitrary $T>0$, because we already know by Theorem~\ref{SVA} that the difference solutions are bounded up to $T$.
We introduce the following notation with the $\lambda_1$ in Theorem \ref{SVA}:
\begin{eqnarray*}
&&u^\ast:=\sup_{x,t\in \T,c\in[c_0,c_1],|\xi|\le\lambda_1^{-1}}|L^{(c)}_\xi(x,t,\xi)|\,\mbox{ (note that $|u^k_m|\le u^\ast$)},\\
&&H_{xx}^\ast:=\sup_{x,t\in \T,c\in[c_0,c_1],|u|\le u^\ast}|H_{xx}(x,t,c+u)|,\,\,\,H_{xp}^\ast:=\sup_{x,t\in \T,c\in[c_0,c_1],|u|\le u^\ast}|H_{xp}(x,t,c+u)|,\\
&&H_{pp}^\ast:=\inf_{x,t\in \T,c\in[c_0,c_1],|u|\le u^\ast}|H_{pp}(x,t,c+u)|\mbox{ \,\,\,($H^\ast_{pp}>0$ due to (A2))},\\
&&\eta:=\max\{2H_{xp}^\ast+H_{pp}^\ast, \frac{1}{2}H^\ast_{pp}+H^\ast_{xx}\},\quad \,\,\,\,\,\,\,\,\,\,\,\,\,E^\ast:=\frac{2H^\ast_{xp}}{H^\ast_{pp}}+\sqrt{4\Big(\frac{H^\ast_{xp}}{H^\ast_{pp}}\Big)^2+\frac{2H^\ast_{xx}}{H_{pp}^\ast}}.
\end{eqnarray*}
Before giving details, we summarize our strategy as follows:
We estimate $E^{k+1}_\Delta$ from $E^k_\Delta$ by using the difference equation.
We find that each $E^{k+1}_\Delta-E^k_\Delta$ is bounded from above by $P(E^k_\Delta)$, where $P(y)$ is a concave parabola whose zero point on the right-hand side is $E^\ast$; i.e., $P(y)>0$ for $0\le y<E^\ast$, $P(E^\ast)=0$, and $P(y)<0$ for $y>E^\ast$  (see (\ref{289}) below).
Hence, if $E^k_\Delta>E^\ast$ (resp. $E^k_\Delta<E^\ast$), then $E^{k+1}_\Delta$ decreases by at least $P(E^k_\Delta)<0$ (resp. increases by at most $P(E^k_\Delta)>0$) and can remain near $E^\ast$ for large $k\le k(T)$.
If $E^0_\Delta$ is very large, $E^k_\Delta$ decays rapidly at first in a way similar to that of solutions to $w'(s)=-(w(s))^2$, where $w(s)\sim 1/s$.
\begin{Prop}\label{entropy-condition}
Let $\lambda_1>0$ be that of Theorem~\ref{SVA}.
Suppose that $\Delta=(\dx,\dt)$ satisfies $\lambda=\dt/\dx<\lambda_1$, $\dt<\min\{(2\eta)^{-1},(E^\ast H^\ast_{pp}+2H^\ast_{xp})^{-1}\}$,
\begin{eqnarray}\label{288}
\sup_{x,t\in \T,c\in[c_0,c_1],|u|\le u^\ast}\lambda (|H_p(x,t,c+u)|+H^\ast_{xp}\cdot2\dx)<1,\quad \lambda\le\frac{1-2H^\ast_{xp}\dt}{rH_{pp}^\ast+(1+H^\ast_{pp})\dx}.
\end{eqnarray}
Then, the following hold:
\begin{enumerate}
 \item For $1\le k\le k(T)$ we have
$$E^k_\Delta=\max_{m}\frac{u^k_{m+2}-u^k_m}{2\dx}\le \frac{2e^{\eta t_k}}{H^\ast_{pp}}\frac{1}{t_k}\mbox{ \,\,\,($t_k=k\dt$)}.$$
 \item If $E^0_\Delta\le E^\ast$, we have $E^k_\Delta\le E^\ast$ for $1\le k\le k(T)$.
 \item If $k> k(\eta^{-1})$, we have $\dis E^k_\Delta\le \frac{4e\eta}{H^\ast_{pp}}$.
 \item If $u^k_m$ is extended to $k\to\infty$ with $|u^k_m|\le u^\ast$, we have $\dis \limsup_{k\to\infty}E^k_\Delta\le E^\ast.$
\end{enumerate}
\end{Prop}
\begin{proof}[{\bf Proof.}]
Using the difference equation and Taylor's formula, we obtain an estimate of $E^{k+1}_\Delta$ from $E^k_\Delta$.
For brevity, the remainders in Taylor's formula are denoted by $H_{pp}$, $H_{xx}$, and $H_{xp}$, which satisfy
$$H_{pp}\ge H^\ast_{pp},\,\,\,|H_{xx}|\le H^\ast_{xx},\,\,\,|H_{xp}|\le H^\ast_{xp}.$$
Set $z^{k}_m:=u^k_{m+2}-u^k_m$.
Then, we have
\begin{eqnarray*}
z^{k+1}_{m+1}&=&\frac{z^k_m+z^k_{m+2}}{2}-\frac{\dt}{2\dx}
\{H(x_{m+4},t_k,c+u^k_{m+4})-H(x_{m+2},t_k,c+u^k_{m+4})\\
&&+H(x_{m+2},t_k,c+u^k_{m+4})-H(x_{m+2},t_k,c+u^k_{m+2})\\
&& +H(x_{m},t_k,c+u^k_{m})-H(x_{m+2},t_k,c+u^k_{m})\\
&&+H(x_{m+2},t_k,c+u^k_{m})-H(x_{m+2},t_k,c+u^k_{m+2})\}\\
&=& (\frac{1}{2}+\frac{\lambda}{2}H_p(x_{m+2},t_k,c+u^k_{m+2}))z^k_m+
(\frac{1}{2}-\frac{\lambda}{2}H_p(x_{m+2},t_k,c+u^k_{m+2}))z^k_{m+2}\\
&&-\frac{\dt}{2\dx}\{ (H_x(x_{m+2},t_k,c+u^k_{m+4})-H_x(x_{m+2},t_k,c+u^k_{m}))(2\dx)\\
&&\qquad\qquad+\frac{1}{2}H_{pp}\cdot(z^k_{m+2})^2+\frac{1}{2}H_{pp}\cdot(z^k_m)^2+\frac{1}{2}H_{xx}\cdot(2\dx)^2 +\frac{1}{2}H_{xx}\cdot(2\dx)^2 \}\\
&=& \{\frac{1}{2}+\frac{\lambda}{2}H_p(x_{m+2},t_k,c+u^k_{m+2})-\frac{\lambda}{2}H_{xp}\cdot2\dx\}z^k_m\\
&&+\{\frac{1}{2}-\frac{\lambda}{2}H_p(x_{m+2},t_k,c+u^k_{m+2})-\frac{\lambda}{2}H_{xp}\cdot2\dx\}z^k_{m+2}\\
&&-\frac{\dt}{2\dx}\{ \frac{1}{2}H_{pp}\cdot(z^k_{m+2})^2+\frac{1}{2}H_{pp}\cdot(z^k_m)^2+\frac{1}{2}H_{xx}\cdot(2\dx)^2 +\frac{1}{2}H_{xx}\cdot(2\dx)^2 \}.\\
\end{eqnarray*}
By the first inequality in (\ref{288}), it follows that
$$\{\frac{1}{2}\pm\frac{\lambda}{2}H_p(x_{m+2},t_k,c+u^k_{m+2})-\frac{\lambda}{2}H_{xp}\cdot2\dx\}>0.$$
Hence, setting $\tilde{z}^k_m:=\max\{z^k_m,z^k_{m+2}\}$, we obtain
\begin{eqnarray*}
z^{k+1}_{m+1}&\le& (1-2H_{xp}\dt)\tilde{z}^k_m+H^\ast_{xx}\cdot2\dx\dt-\frac{H^\ast_{pp}}{2}\frac{\dt}{2\dx}(\tilde{z}^k_m)^2,\\
\frac{z^{k+1}_{m+1}}{2\dx}&\le&(1-2H_{xp}\dt) \frac{\tilde{z}^k_m}{2\dx}+H^\ast_{xx}\dt-\frac{H^\ast_{pp}}{2}\dt(\frac{\tilde{z}^k_m}{2\dx})^2.
\end{eqnarray*}
Note that $g(y):=(1-2H_{xp}\dt)y+H^\ast_{xx}\dt-(\frac{H^\ast_{pp}}{2}\dt )y^2$ is monotonically increasing if
$$y\le \frac{1-2H_{xp}\dt}{H_{pp}^\ast\dt}.$$
From the second inequality in (\ref{288}) it follows that $\lambda\le(1-2H^\ast_{xp}\dt)/(rH_{pp}^\ast+\dx)$ and hence
$$E^0_\Delta\le \frac{2r}{2\dx}\le \frac{1-2H_{xp}^\ast\dt}{H_{pp}^\ast\dt}\le \frac{1-2H_{xp}\dt}{H_{pp}^\ast\dt}$$
for all initial data in $L^\infty_{r,0}(\T)$.
Suppose that $E^k_\Delta\le (1-2H_{xp}^\ast\dt)/(H_{pp}^\ast\dt)$.
Then, we obtain
\begin{eqnarray}\label{289}
E^{k+1}_\Delta\le E^k_\Delta+P(E^k_\Delta),\,\,\,P(y):=-\dt(\frac{H^\ast_{pp}}{2}y^2-2H^\ast_{xp}y-H^\ast_{xx}).
\end{eqnarray}
From $\dt<(E^\ast H^\ast_{pp}+2H^\ast_{xp})^{-1}$ it follows that $E^\ast< (1-2H_{xp}^\ast\dt)/(H_{pp}^\ast\dt)$, and from $\dt<(2\eta)^{-1}$ it follows that $|y-E^\ast|\ge |P(y)|$ for all $0\le y\le (1-2H_{xp}^\ast\dt)/(H_{pp}^\ast\dt)$.
Hence, we have two cases:
\begin{itemize}
\item[(1)] If $E^k_\Delta \le E^\ast$, we may have $E^{k+1}_\Delta\ge E^k_\Delta$, but we certainly have $E^{k+1}_\Delta\le E^k_\Delta+P(E^k_\Delta)\le E^\ast$.
\item[(2)] If $E^\ast < E^k_\Delta$, we have $E^{k+1}_\Delta<E^k_\Delta$.
\end{itemize}
\noindent Therefore, we have $E^{k+1}_\Delta\le (1-2H_{xp}^\ast\dt)/(H_{pp}^\ast\dt)$ and, by induction, it follows that $E^k_\Delta\le (1-2H_{xp}^\ast\dt)/(H_{pp}^\ast\dt)$ for all $0\le k\le k(T)$.
Thus, (\ref{289}) holds for all $0\le k< k(T)$.
It is now easy to verify that 1) if $E^0_\Delta\le E^\ast$, then $E^k_\Delta$ may increase but never exceed $E^\ast$, and 2) if $E^0_\Delta>E^\ast$, then the $E^k_\Delta$ are bounded from above by a monotonically decreasing sequence. Claim 4 is also clear. 

Now we follow Lemma~2 in~\cite{Oleinik}.
Set $V^k:=E^k_\Delta+1\ge1$.
Then, by (\ref{289}) we have
$$V^{k+1}\le(1+\eta\dt)V^k-\frac{H^\ast_{pp}}{2}\dt(V^k)^2$$
We set $W^k:=(1-\eta\dt)^kV^k$ for $k\ge0$ ($1-\eta\dt>0$ holds since $\dt<(2\eta)^{-1}$).
Then, for $k\ge1$ we have
\begin{eqnarray*}
W^{k+1}&\le& (1-\eta\dt)(1+\eta\dt)W^k-\frac{H^\ast_{pp}}{2}\dt (W^k)^2(1-\eta\dt)^{-k+1}\\
&\le& W^k-\frac{H^\ast_{pp}}{2}\dt (W^k)^2.
\end{eqnarray*}
Consider $w'(t)=-\frac{H^\ast_{pp}}{2} (w(t))^2$, with $w(0)=w^0:=2/(H^\ast_{pp}\dt)$.
The solution satisfies
$$w(t)=\frac{1}{\frac{H^\ast_{pp}}{2}t+\frac{1}{w^0}}\le \frac{2}{H^\ast_{pp}t}.$$
We can show that $W^k\le w(k\dt)$ for $k\ge 1$ by noting that $w(\dt)=1/(H^\ast_{pp}\dt)$ and $W^1=(1-\eta\dt)(E^1_\Delta+1)\le r/\dx+1$.
From the second inequality in (\ref{288}) it follows that $\lambda\le 1/\{(r+\dx)H^\ast_{pp}\}$ and hence that $W^1\le w(\dt)$.
Suppose that $W^k\le w(k\dt)$ for some $k\ge1$.
Then, since $g(y):=y-\frac{H^\ast_{pp}\dt}{2}y^2$ is monotonically increasing for $y\le 1/(H^\ast_{pp}\dt)$,  $w(k\dt)\le 1/(H^\ast_{pp}\dt)$, and $w''>0$,  we have
\begin{eqnarray*}
W^{k+1}&\le& W^k-\frac{H^\ast_{pp}\dt}{2}(W^k)^2\le w(k\dt)-\frac{H^\ast_{pp}\dt}{2}(w(k\dt))^2\\
&=&w(k\dt)+\dt w'(k\dt)=w(k\dt+\dt)-\frac{1}{2}w''(k\dt+\theta\dt)\cdot(\dt)^2\\
&\le&w((k+1)\dt)\mbox{ ($\theta\in(0,1)$)}.
\end{eqnarray*}
Thus, we obtain
$$E^k_\Delta\le(1-\eta\dt)^{-k}\frac{2}{H^\ast_{pp}k\dt}\le(1-\eta\dt)^{-\frac{\eta k\dt}{\eta\dt}}\frac{2}{H^\ast_{pp}k\dt}\le \frac{2e^{\eta t_k}}{H^\ast_{pp}}\frac{1}{t_k}.$$
\indent Setting $f(t):=\frac{2e^{\eta t}}{H^\ast_{pp}}\frac{1}{t}$,
the minimum of $f$ becomes $f(\eta^{-1})=\frac{2e\eta}{H^\ast_{pp}}$, which is greater than $E^\ast$.
Therefore, due to Cases~(1) and~(2), the $E^k_\Delta$ are bounded from above by $f(t_k)(>E^\ast)$ for $k\le k(\eta^{-1})$ and never exceed $f(\eta^{-1}-\dt)\le \frac{4e\eta}{H^\ast_{pp}}$ for $k>k(\eta^{-1})$.
This demonstrates the proposition.
\end{proof}
\setcounter{section}{2}
\setcounter{equation}{0}
\section{Time-Global Stability and Large-Time Behavior}
We prove time-global stability of the Lax-Friedrichs scheme with a fixed mesh size.
Then, we show the large-time behavior of the scheme in which each difference solution falls into a time periodic state with unit period. 
Each time periodic state corresponds to a space-time periodic difference solution. 
There arises the notion of the effective Hamiltonian of (\ref{HJ-Delta}).
\subsection{Time-Global Stability}
The main result of this section is the following theorem.
\begin{Thm}\label{Main-stability}
There exist $\lambda_1>0$ and $\delta>0$ such that, if $\Delta=(\dx,\dt)$ satisfies $0<\lambda_0\le \lambda=\dt/\dx<\lambda_1$ and $\dx\le\delta$, the Lax-Friedrichs scheme starting from any $u^0\in L^\infty_{r,0}(\T)$ succeeds up to an arbitrary time index and satisfies the CFL~condition
$$|H_p(x_m,t_k,c+u^k_m)|\le \lambda_1^{-1}<\lambda^{-1}\mbox{ \,\,\,for all $m\in\Z$ and $k\in\Z_+$.}$$
\end{Thm}
\noindent In order to prove this theorem, we need uniform boundedness of $\norm \phi^1_\Delta(u^0;c)\norm_{L^\infty}$ with respect to $(u^0;c)$  similar to that in Proposition~\ref{boundedness}.
First, we observe the following lemma.
\begin{Lemma}\label{unif-L^1}
Let $\lambda_1>0$ be that of Theorem~\ref{SVA} and let $\Delta=(\dx,\dt)$ be such that $0<\lambda_0\le\lambda=\dt/\dx<\lambda_1$.
Fix $t\in(0,T]$ arbitrarily.
Then, for any $\varepsilon>0$ there exists $\delta=\delta(\varepsilon;t)>0$ such that, if $\dx\le\delta$, we have
$$\sup_{u^0\in L^\infty_{r,0}(\T),c\in[c_0,c_1]}\norm\phi^t_\Delta(u^0;c)-\phi^t(u^0;c)\norm_{L^1(\T)}\le\varepsilon.$$
\end{Lemma}
\begin{proof}[\bf Proof.]
If not, then for some $\varepsilon_0>0$ and $\delta_j\to0$ as $j\to\infty$, we have $\dx_j\le\delta_j$ such that
\begin{eqnarray}\label{unif-L^1-1}
\sup_{u^0\in L^\infty_{r,0}(\T),c\in[c_0,c_1]}\norm\phi^t_{\Delta_j}(u^0;c)-\phi^t(u^0;c)\norm_{L^1(\T)}>\varepsilon_0,
\end{eqnarray}
where $\Delta_j=(\dx_j,\lambda\dx_j)$.
We show that for each $j$ there exists $(u^0,c)\in L^\infty_{r,0}(\T)\times[c_0,c_1]$ that attains the supremum~(\ref{unif-L^1-1}) denoted by $b$.
Let $(u^0_i,c^i)$ be the sequence for which $\norm\phi^t_{\Delta_j}(u^0_i;c^i)-\phi^t(u^0_i;c^i)\norm_{L^1(\T)}\to b$ as $i\to\infty$.
Let $v^0_i$ be a primitive of $u^0_i$ that belongs to $Lip_r(\T)$ and is bounded by $r$.
By the Arzela-Ascoli theorem we have a subsequence of $(v^0_i,c^i)$, still denoted by $(v^0_i,c^i)$, which converges to $(v^0,c)$.
By Propositions~\ref{continuity} and~\ref{continuity-Delta}, we have
$\phi^t_{\Delta_j}(v^0_i{}_x;c^i)\to \phi^t_{\Delta_j}(v^0_x;c)$ in $L^1(\T)$ and $\phi^t(v^0_i{}_x;c^i)\to \phi^t(v^0_x;c)$ in $L^1(\T)$ as $i\to\infty$.
Therefore, we obtain $\norm\phi^t_{\Delta_j}(v^0_x;c)-\phi^t(v^0_x;c)\norm_{L^1(\T)}= b$.

Let $(u^0_j,c^j)$ attain the supremum~(\ref{unif-L^1-1}).
Let $v^0_j$ be a primitive of  $u^0_j$ that belongs to $Lip_r(\T)$ and is bounded by $r$.
We have a subsequence of $(v^0_j,c^j)$, still denoted by $(v^0_j,c^j)$, which converges to $(v^0,c)$.
It follows from Claim~7 of Theorem~\ref{SVA} that there exists $\delta_0>0$ such that, if $\dx\le\delta_0$, we have $\norm\phi^t_{\Delta}(v^0_x;c)-\phi^t(v^0_x;c)\norm_{L^1(\T)}<\frac{\varepsilon_0}{2}$.
Hence,
\begin{eqnarray*}
\frac{\varepsilon_0}{2}&>&\norm\phi^t_{\Delta}(v^0_x;c)-\phi^t(v^0_x;c)\norm_{L^1(\T)}\\
&\ge&\norm\phi^t_{\Delta}(v^0_j{}_x;c^j)-\phi^t(v^0_j{}_x;c^j)\norm_{L^1(\T)}-\norm\phi^t_{\Delta}(v^0_x;c)-\phi^t_\Delta(v^0_j{}_x;c^j)\norm_{L^1(\T)}\\
&&-\norm\phi^t(v^0_j{}_x;c^j)-\phi^t(v^0_x;c)\norm_{L^1(\T)}.
\end{eqnarray*}
By Propositions~\ref{continuity} and~\ref{continuity-Delta}, we have $\norm\phi^t_{\Delta}(v^0_x;c)-\phi^t_\Delta(v^0_j{}_x;c^j)\norm_{L^1(\T)}+\norm\phi^t(v^0_j{}_x;c^j)-\phi^t(v^0_x;c)\norm_{L^1(\T)}\le \frac{\varepsilon_0}{2}$ for large $j$.
Therefore, we have
$\norm\phi^t_{\Delta}(v^0_j{}_x;c^j)-\phi^t(v^0_j{}_x;c^j)\norm_{L^1(\T)}<\varepsilon_0$ for any $\dx\le\delta_0$, which is a contradiction.
\end{proof}
Next, we see that the convergence  $\norm\phi^1_{\Delta}(u^0;c)-\phi^1(u^0;c)\norm_{L^1(\T)}\to0$ as $\Delta\to0$, which is uniform with respect to $(u^0,c)$,  yields uniform boundedness of $\norm \phi^1_{\Delta}(u^0;c)\norm_{L^\infty}$ with the aid of the one-sided Lipschitz condition.
\begin{Prop}\label{barrier}
Let $\lambda_1>0$ be that of Theorem~\ref{SVA} with $T=1$.
Let $\Delta=(\dx,\dt)$ be such that $0<\lambda_0\le\lambda=\dt/\dx<\lambda_1$, satisfying the conditions in Proposition~\ref{entropy-condition}.
Then, there exists $\delta>0$ such that, if $\dx\le\delta$, with $\beta_1$ of Proposition~\ref{boundedness} we have
$$\sup_{u^0\in L^\infty_{r,0}(\T),c\in[c_0,c_1]}\norm \phi^1_\Delta(u^0;c)\norm_{L^\infty}\le \beta_1(1)+1.$$
\end{Prop}
\begin{proof}[\bf Proof]
Let $0<\varepsilon<1$ be such that $\frac{1-\sqrt{\varepsilon}}{3\sqrt{\varepsilon}}>2e^{\eta}/H_{pp}^\ast\ge E^{2K}_\Delta$, where $2K\dt=1$.
With this $\varepsilon$ and $t=1$, we have $\delta>0$ in Lemma~\ref{unif-L^1}.
We take $\dx\le \min\{\delta,\sqrt{\varepsilon}\}$.
Consider
$$A:=\{y\in\T\,\,|\,\,|\phi^1_\Delta(u^0;c)(y)-\phi^1(u^0;c)(y)|> \sqrt{\varepsilon}\}.$$
Since $\norm\phi^1_\Delta(u^0;c)-\phi^1(u^0;c)\norm_{L^1(\T)}\le\varepsilon$, we have $\mes[A]\le\sqrt{\varepsilon}$.
Hence, for $y\in A$ there exists $x,\tilde{x}\in\R\setminus A$ such that $0<y-x\le \sqrt{\varepsilon}$ and $0<\tilde{x}-y\le \sqrt{\varepsilon}$.
For $x\in\T\setminus A$, we have $|\phi^1_\Delta(u^0;c)(x)-\phi^1(u^0;c)(x)|\le\sqrt{\varepsilon}$ and $|\phi^1_\Delta(u^0;c)(x)|\le|\phi^1(u^0;c)(x)|+ \sqrt{\varepsilon}\le \beta_1(t)+1$.
Consider
$$\tilde{A}:=\{ y\in A\,\,|\,\,|\phi^1_\Delta(u^0;c)(y)|> \beta_1(t)+1 \}.$$
Suppose that $\tilde{A}$ is not empty.
Then, there exists $x_n\in \tilde{A}\cap(\dx\Z)$ such that $u^{2K}_n>  \beta_1(1)+1$ (resp. $u^{2K}_n<-\beta_1(1)-1$).
Since there exist $x_m,x_{m'}\in(\R\setminus A)\cap(\dx\Z)$ such that $0<x_n-x_m\le \sqrt{\varepsilon}+2\dx\le3\sqrt{\varepsilon}$ and $0<x_{m'}-x_n\le \sqrt{\varepsilon}+2\dx\le3\sqrt{\varepsilon}$,  we have
\begin{eqnarray*}
&&\frac{u^{2K}_n-u^{2K}_m}{x_n-x_m}>\frac{\beta_1(1)+1-(\beta_1(1)+\sqrt{\varepsilon})}{3\sqrt{\varepsilon}}=\frac{1-\sqrt{\varepsilon}}{3\sqrt{\varepsilon}}> E^{2K}_\Delta,\\
\mbox{\Big(resp.}&&\frac{u^{2K}_{m'}-u^{2K}_n}{x_{m'}-x_n}>\frac{-(\beta_1(1)+\sqrt{\varepsilon})-(-\beta_1(1)-1)}{3\sqrt{\varepsilon}}=\frac{1-\sqrt{\varepsilon}}{3\sqrt{\varepsilon}}> E^{2K}_\Delta\Big).
\end{eqnarray*}
These two inequalities contradict the one-sided Lipschitz condition.
\end{proof}
\noindent{\bf Remark.} Claim~7 of Theorem~\ref{SVA} states that $\phi^1_\Delta(u^0;c)$ converges to $\phi^1(u^0;c)$ uniformly on $\T\setminus\Theta$ as $\Delta\to0$, where $\Theta$ is an arbitrary small neighborhood of shocks.
However, we cannot use this fact for Proposition~\ref{barrier}, because uniformity of the convergence with respect to $(u^0;c)$ is unverified.

\medskip

\begin{proof}[\bf Proof of Theorem~\ref{Main-stability}.]
Let $\lambda_1>0$ be that of Theorem~\ref{SVA} with $T=1$ and $r\ge\beta_1(1)+1$.
Let $\delta>0$ be that of Proposition~\ref{barrier}.
Then, $\tilde{u}^0:=\phi^1_\Delta(u^0;c)$ belongs to $L^\infty_{\beta_1(1)+1,0}(\T)$ for any $u^0\in L^\infty_{r,0}(\T)$.
Hence, by the choice of $\lambda_1$, we are guaranteed that $\phi^1_\Delta(\tilde{u}^0;c)=\phi^2_\Delta(u^0;c)$ is well defined and bounded by $\beta_1(1)+1$ again.
In this way, $\phi^t_\Delta(u^0;c)$ can be defined for $t\to\infty$ with the CFL~condition.
\end{proof}
\subsection{Large-Time Behavior}
If we take $r\ge\beta_1(1)+1$, then $\phi^1_\Delta(u^0;c)$ belongs to  $L^\infty_{r,0}(\T)$.
Therefore, $\phi^1_\Delta(\cdot;c)$ maps $L^\infty_{r,0}(\T)$ into itself.
We can find the fixed points of the map for each $c$.
In this subsection, we consider the fixed points and their stability, which makes clear the large-time behavior of the Lax-Friedrichs scheme.
Note that the Lax-Friedrichs scheme has a contraction property under the CFL~condition.
That is, for $0\le t\le t'$ we have
\begin{eqnarray*}
\norm \phi^{t'}_\Delta({u}^0;c)-\phi^{t'}_\Delta(\tilde{u}^0;c)\norm_{L^1(\T)}&\le& \norm \phi^t_\Delta({u}^0;c)-\phi^t_\Delta(\tilde{u}^0;c)\norm_{L^1(\T)}.
\end{eqnarray*}
This can be refined to become a strict contraction property.
Let $\sum_{m;k}$ denote summation with respect to $\{m\,|\,0\le m<2N,\,\,m+k\;\mbox{even}\}$ for each fixed $k$, and let $\norm x \norm_1:=\sum_{1\le j\le n}|x_j|$ for $x\in\R^n$.
\begin{Prop}\label{contraction}
The family of maps $\{\phi^t_\Delta(\cdot;c)\}_{t\ge0}$ has a strict contraction property within the unit time period.
That is, for any two distinct initial data $u^0$ and $\tilde{u}^0$, we have
$$\norm \phi^{t+1}_\Delta(u^0;c)-\phi^{t+1}_\Delta(\tilde{u}^0;c)\norm_{L^1(\T)} < \norm \phi^{t}_\Delta(u^0;c)-\phi^{t}_\Delta(\tilde{u}^0;c)\norm_{L^1(\T)}.$$
\end{Prop}
\begin{proof}[{\bf Proof.}]
It is sufficient to show that for all $k\ge 0$ any two difference solutions $u^k$ and $\tilde{u}^k$ of (\ref{2CL-Delta}) satisfy
$$\norm u^{k+2K}-\tilde{u}^{k+2K}\norm_1<\norm u^k-\tilde{u}^k\norm_1.$$
Set $z^k_m:=u^k_m-\tilde{u}^k_m$ and $\sigma^k_m:=\mbox{sign\,}z^k_m=1$ or $-1$ (${\rm sign}\,0:=1$).
Then, $\norm u^k-\tilde{u}^k\norm_1=\sum_{m;k}|z^{k}_m|=\sum_{m;k}\sigma^k_m z^{k}_m$.
By the difference equation of (\ref{2CL-Delta}), we have
\begin{eqnarray*}
\sum_{m;k}|z^{k+1}_{m+1}|=\sum_{m;k}\sigma^{k+1}_{m+1}z^{k+1}_{m+1}=\sum_{m;k}\sigma^{k+1}_{m+1}\Big\{\frac{1}{2}z^k_{m+2}(1-\lambda\delta^k_{m+2})+\frac{1}{2}z^k_{m}(1+\lambda\delta^k_{m}) \Big\},
\end{eqnarray*}
where $\delta^k_m:=H_p(x_{m},t_k,c+u^k_m+\theta^k_m)$ with a constant $\theta^k_m$ derived from Taylor's formula.
Switching the order of the summations above, we obtain
\begin{eqnarray*}
\sum_{m;k}|z^{k+1}_{m+1}|&=&\sum_{m;k}z^{k}_{m}\left\{\frac{1}{2}\sigma^{k+1}_{m-1}(1-\lambda\delta^k_{m})+\frac{1}{2}\sigma^{k+1}_{m+1}(1+\lambda\delta^k_{m}) \right\}\\
&=&\sum_{m;k}|z^k_m|+\sum_{m;k}|z^k_m|\left[-1+ \sigma^{k}_{m}
\left\{\frac{1}{2}\sigma^{k+1}_{m-1}(1-\lambda\delta^k_{m})+\frac{1}{2}\sigma^{k+1}_{m+1}(1+\lambda\delta^k_{m}) \right\}
\right].
\end{eqnarray*}
Let $R^k$ denote the second sum in the second line of the above equality.
We find that $R^k\le0$, since for each term of $R^k$ the factor $[\quad]$ belongs to one of two cases:

(1)\quad If $\sigma^{k+1}_{m-1}+\sigma^{k+1}_{m+1}=\pm2$, then $[\quad]=-1\pm\sigma^k_m=0$ or $-2$.

(2)\quad If $\sigma^{k+1}_{m-1}+\sigma^{k+1}_{m+1}=0$, then $[\quad]=-1\pm\lambda\delta^k_m<0$ due to the CFL~condition.

\noindent Since $u^k$ and $\tilde{u}^k$ each have zero mean and $u^0\neq \tilde{u}^0$,  the sign of $z^k_m$ necessarily changes and Case~(2) occurs.
It seems possible that even though $u^k$ and $\tilde{u}^k$ are such, we may have $R^k=0$; namely, $z^{k}_{m}=0$ for all the integers $m$ for which Case~(2) occurs.
However, after further $k^\ast$-time evolution ($k^\ast<N<2K$), Case~(2) certainly occurs and $R^{k+k^\ast}<0$, because such zero-points disappear as $k$ increases due to the monotonicity of the Lax-Friedrichs scheme under the CFL~condition (see also Remark~2.5 in~\cite{Nishida-Soga}).
\end{proof}
\indent We now show that time periodic difference solutions not only exist but are stable, which provides the large-time behavior of the Lax-Friedrichs scheme.
\begin{Thm}\label{fixed-point1}
Take $r\ge\beta_1(1)+1$ and fix $\Delta=(\dx,\dt)$ so that Theorems~\ref{SVA} and~\ref{Main-stability} hold.
Then, for each $c$ there exists a fixed point $\bar{u}^0_\Delta\in L^\infty_{r,0}(\T)$ of $\phi_\Delta^1(\cdot;c)$, which yields a time periodic difference solution $\phi^t_\Delta(\bar{u}^0_\Delta;c)$.
Such a periodic solution is unique with respect to $c$.
Any other solution $\phi^t_\Delta(u^0;c)$ exponentially falls into the periodic state; namely, there exist $\rho\in(0,1)$ and $a>0$ depending on $\Delta$, but independent of $u^0$, such that
$\norm\phi^t_\Delta(u^0;c)-\phi^t_\Delta(\bar{u}^0_\Delta;c)\norm_{L^\infty}\le a\rho^t\mbox{ for $t\in\N$}.$
\end{Thm}
\begin{proof}[{\bf Proof.}]
The map $\phi^1_\Delta(\cdot;c)$ is actually a map from $\R^N$ to $\R^N$, since the step functions have only $N$ different values at most.
Let $B_r$ be the set of all $x\in\R^N$ with $\norm x\norm_\infty\le r$.
If $r\ge\beta_1(1)+1$, then  the map $\phi^1_\Delta(\cdot;c)$ is actually a map from $B_r$ to $B_r$.
Therefore, we obtain a fixed point through Brouwer's fixed point theorem.
By Proposition~\ref{contraction}, periodic solutions must be unique.
Exponential decay can be proved in the same way as (5) of Theorem~2.1 in \cite{Nishida-Soga}.
\end{proof}
\medskip

\noindent {\bf Remark.} It is likely that in general $\rho$ becomes arbitrarily close to unity as $\Delta$ tends to zero. Numerical experiments imply such a property of $\rho$ \cite{Nishida-Soga}. The uniqueness does not hold for the exact equation (\ref{CL2}) in general. There may exist time periodic entropy solutions of~(\ref{CL2}) with the minimum period greater than unity \cite{Bernard}.
\medskip

\noindent  The following theorem for the discrete Hamilton-Jacobi equation is like the weak KAM theorem.
\begin{Thm}\label{fixed-point2}
Take $r\ge\beta_1(1)+1$ and fix $\Delta=(\dx,\dt)$ so that Theorems~\ref{SVA} and~\ref{Main-stability} hold.
Then, for each $c$ there exists a constant $\bar{h}_\Delta(c )\in\R$ such that if $h(c )=\bar{h}_\Delta(c )$, we have a fixed point $\bar{v}^0_\Delta\in Lip_r(\T)$ of $\psi_\Delta^1(\cdot;c)$, which yields a time periodic difference solution $\psi^t_\Delta(\bar{v}^0_\Delta;c)$.
Such a periodic solution is unique with respect to $c$ up to constants.
Any other solution $\psi^t_\Delta(v^0;c)$ exponentially falls into a periodic state; namely, for the $\rho\in(0,1)$ and $a>0$ in Theorem~\ref{fixed-point1} and for $d\in\R$ depending on $(v^0;c)$ we have
$\norm\psi^t_\Delta(v^0;c)-\psi^t_\Delta(\bar{v}^0_\Delta+dI_1;c)\norm_{C^0}\le a\rho^t\mbox{ for $t\in\N$}$, where $I_1(x):=1$ and $\psi^t_\Delta(v^0+dI_1;c)=\psi^t_\Delta(v^0;c)+dI_1$.
\end{Thm}
\begin{proof}[{\bf Proof.}] We imitate the proof of the weak KAM theorem \cite{Fathi}.
Let us write $v\sim w$ for $v,w\in C^0(\T)$ if there exists $b\in\R$ such that $w=v+bI_1$.
We introduce $\hat{v}:=\{w\in C^0(\T)\,|\,w\sim v\}$, $\norm \hat{v}\norm:=\inf_{w\in\hat{v}}\norm w\norm_{C^0(\T)}$, $\hat{C^0(\T)}:=C^0(\T)/\sim$, and $\hat{Lip_r(\T)}:=Lip_r(\T)/\sim$.
From the Arzela-Ascoli theorem it follows that $\hat{Lip_r(\T)}$ is a compact convex subset of the Banach space $\hat{C^0(\T)}$.
Due to the property $\psi^t_\Delta(v^0+dI_1;c)=\psi^t_\Delta(v^0;c)+dI_1$, we have
$$\psi^1_\Delta(v^0;c)\sim\psi^1_\Delta(w^0;c)\mbox{ for all $v^0,w^0\in\hat{v}^0$}.$$
Hence, the map
$$\hat{\psi}^1_\Delta(\cdot;c):\hat{Lip_r(\T)}\to\hat{Lip_r(\T)},\,\,\,\hat{\psi}^1_\Delta(\hat{v}^0;c):=\{w\in Lip_r(\T)\,|\,w\sim \psi^1_\Delta(v^0;c)\} \,\,\,(v^0\in\hat{v}^0)$$
is well defined and continuous.
By Schauder's fixed point theorem, we have a fixed point $\hat{\bar{v}}^0_\Delta$ satisfying
$\hat{\psi}^1_\Delta(\hat{\bar{v}}^0_\Delta;c)=\hat{\bar{v}}^0_\Delta$.
Therefore, we have an element $\bar{v}^0_\Delta\in\hat{\bar{v}}^0_\Delta$ and  constant $b(c )\in\R$ such that
$$\bar{v}^0_\Delta=\psi^1_\Delta(\bar{v}^0_\Delta;c)+ b(c )I_1.$$
This relation means that $\bar{v}^0_\Delta$ yields a time periodic solution of  (\ref{2HJ-Delta}) with $h(c)+b(c)$ instead of $h(c )$.

Note that $\psi^t_\Delta(v^0;c)\le\psi^t_\Delta(\tilde{v}^0;c)$ if $v^0\le\tilde{v}^0$.
Let $a^0,b^0$ be constants such that for all $x\in\T$ we have
$$\bar{v}^0_\Delta(x)+b^0\le v^0(x)\le \bar{v}^0_\Delta(x)+a^0,$$
with at least one point attaining the equality in each inequality.
Then, we have $\bar{v}^0_\Delta(x)+b^0\le\psi^1_\Delta(v^0;c)(x)\le \bar{v}^0_\Delta(x)+a^0$ for all $x\in\T$.
Let $a^1,b^1$ be constants such that for all $x\in\T$ we have
$$\bar{v}^0_\Delta(x)+b^1\le \psi^1_\Delta(v^0;c)(x) \le \bar{v}^0_\Delta(x)+a^1,$$
with at least one point attaining the equality in each inequality.
Note that $a^1\le a^0$ and $b^1\ge b^0$.
Then, we have $\bar{v}^0_\Delta(x)+b^1\le\psi^2_\Delta(v^0;c)(x)\le \bar{v}^0_\Delta(x)+a^1$ for all $x\in\T$.
In this way, we obtain the bounded monotone sequences $a^j$ and $b^j$.\
Take $d$ such that $\lim_{j\to\infty} b^j\le d\le \lim_{j\to\infty} a^j$.
Then, $\bar{v}^0_\Delta+dI_1$ and $ \psi^t_\Delta(v^0;c)$ coincide for at least one point and for any $t\in\N$.
Let $x_0\in\T$ be such that $\bar{v}^0_\Delta(x_0)+d= \psi^t_\Delta(v^0;c)(x_0)$.
Then, for all $x\in\T$ and $t\in\N$ we obtain
$$|\psi^t_\Delta(v^0;c)(x)-\psi^t_\Delta(\bar{v}^0_\Delta+dI_1;c)(x)|\le \Big|\int_{x_0}^x|\phi^t_\Delta(v^0_x;c)-\phi^t_\Delta(\bar{u}^0_\Delta;c)|dy\Big|\le a\rho^t.$$
 \end{proof}
\noindent We introduce the map $\bar{h}_\Delta(c ):c\mapsto h(c )+b(c )$, which is the effective Hamiltonian of the difference Hamilton-Jacobi equation (\ref{HJ-Delta}).
We remark that $\bar{h}_\Delta(c )$ plays an important role in the numerical analysis of the weak KAM theory.
Hence, its properties are meaningful to investigate.
\subsection{Effective Hamiltonian}
Below is the characterization of $\bar{h}_\Delta(c )$, which is very similar to that of the effective Hamiltonian  $\bar{h}(c )$ of the exact Hamilton-Jacobi equations (\ref{HJ2}).
We refer to \cite{Bernard} for the characterization of $\bar{h}(c )$.
\begin{Thm}\label{effective}
\begin{enumerate}
\item $h(c )=\bar{h}_\Delta(c )$ is the unique value for which (\ref{HJ-Delta}) admits a space-time periodic difference solution.
\item $\bar{h}_\Delta( c)$ is the averaged Hamiltonian.
That is, for the space-time periodic difference solution $\bar{u}^k_m$ of (\ref{CL-Delta}) we have
$$\bar{h}_\Delta( c)=\sum_{0\le k<2K}\sum_{m;k}H(x_m,t_k,c+\bar{u}^k_m(c ))\cdot2\dx\dt.$$
\item Let $v^{l+1}_{n}(c )$ be a time-global solution of the difference equation
\begin{eqnarray}\label{373}
D_tv^{k+1}_m+H(x_m,t_k,c+D_xv^k_{m+1})=0.
\end{eqnarray}
Then, for all $n$ we have
$$\lim_{l\to\infty}\frac{v^{l+1}_{n}(c )}{t_{l+1}}=-\bar{h}_\Delta(c ).$$
\item $\bar{h}_\Delta(c )$ is a  convex $C^1$-function.
\item $\bar{h}_\Delta( c)$ uniformly converges to the exact effective Hamiltonian $\bar{h}( c )$ of (\ref{HJ2}) as $\Delta\to0$:
$$\sup_{c\in[c_0,c_1]}|\bar{h}_\Delta( c)-\bar{h}( c)|\le \beta_3\sqrt{\dx}.$$
\end{enumerate}
\end{Thm}
\begin{proof}[{\bf Proof.}]
1. Let $\tilde{\bar{v}}^k_{m+1}$ be another space-time periodic solution of (\ref{HJ-Delta}) with $h(c)=\tilde{\bar{h}}_\Delta(c )$.
Extending the periodic solutions to the entire odd grid, we have the following stochastic and variational representation formulas up to any negative time index $l_0$:
\begin{eqnarray*}
\bar{v}^{l+1}_n&=&E_{\mu(\cdot;\xi^\ast)}\Big[\sum_{l_0<k\le l+1}L^{(c)}(\gamma^k,t_{k-1},\xi^\ast{}^k_{m(\gamma^k)})\dt +\bar{v}^{l_0}_{m(\gamma^{l_0})}\Big]+\bar{h}_\Delta(c )(t_{l+1}-t_{l_0}),\\
\tilde{\bar{v}}^{l+1}_n&=&E_{\mu(\cdot;\tilde{\xi}^\ast)}\Big[\sum_{l_0<k\le l+1}L^{(c)}(\gamma^k,t_{k-1},\tilde{\xi}^\ast{}^k_{m(\gamma^k)})\dt +\tilde{\bar{v}}^{l_0}_{m(\gamma^{l_0})}\Big]+\tilde{\bar{h}}_\Delta(c )(t_{l+1}-t_{l_0}).
\end{eqnarray*}
By the variational property, we have
\begin{eqnarray}\label{3aa}
\tilde{\bar{v}}^{l+1}_n-{\bar{v}}^{l+1}_n\le E_{\mu(\cdot;\xi^\ast)}\Big[\tilde{\bar{v}}^{l_0}_{m(\gamma^{l_0})} - \bar{v}^{l_0}_{m(\gamma^{l_0})}\Big]+(\tilde{\bar{h}}_\Delta(c )-\bar{h}_\Delta(c ))(t_{l+1}-t_{l_0}).
\end{eqnarray}
Note that $\bar{v}^k_{m+1},\tilde{\bar{v}}^k_{m+1}$ are periodic and hence bounded.
Dividing (\ref{3aa}) by $t_{l+1}-t_{l_0}$ and letting $l_0\to-\infty$, we obtain
$$0\le\tilde{\bar{h}}_\Delta(c )-\bar{h}_\Delta(c ).$$
Similar reasoning yields the converse inequality.

2. Since $\bar{v}^k_{m+1}$ satisfies $D_t\bar{v}^{k+1}_{m}+H(x_m,t_k,c+D_x\bar{v}^k_{m+1})=\bar{h}_\Delta(c )$, we have
$$\bar{h}_\Delta(c )= \sum_{0\le k<2K}\sum_{m;k}D_t\bar{v}^{k+1}_{m}\cdot2\dx\dt+\sum_{0\le k<2K}\sum_{m;k}H(x_m,t_k,c+D_x\bar{v}^k_{m+1})\cdot2\dx\dt.$$
The first term on the right-hand side is equal to zero due to the periodicity of $\bar{v}^k_{m+1}$.

3. Let $\tilde{v}^{l+1}_{n}(c )$ be the solution of $D_t\tilde{v}^{k+1}_m+H(x_m,t_k,c+D_x\tilde{v}^k_{m+1})=\bar{h}_\Delta(c )$ with the same mesh size as (\ref{373}) and with $\tilde{v}^{0}_{m+1}=v^0_{m+1}$.
From Theorem~\ref{fixed-point2} it follows that we have $|\tilde{v}^{l+1}_n(c)-\bar{v}^{l+1}_n(c )|\to0$ as $l\to\infty$, adding a constant if necessary.
Therefore, $\tilde{v}^{l+1}_n$ is bounded for $l\to\infty$.
Since
\begin{eqnarray*}
v^{l+1}_n(c )&=&\inf_{\xi}E_{\mu(\cdot;\xi)}\Big[\sum_{0<k\le l+1}L^{(c)}(\gamma^k,t_{k-1},\xi^k_{m(\gamma^k)})\dt +v^0_{m(\gamma^0)}\Big], \\
\tilde{v}^{l+1}_n(c )&=&\inf_{\xi}E_{\mu(\cdot;\xi)}\Big[\sum_{0<k\le l+1}L^{(c)}(\gamma^k,t_{k-1},\xi^k_{m(\gamma^k)})\dt +v^0_{m(\gamma^0)}\Big]+\bar{h}_\Delta(c )t_{l+1},
\end{eqnarray*}
and the minimizing velocity fields of these are the same, we obtain $v^{l+1}_n(c)-\tilde{v}^{l+1}_n(c)=-\bar{h}_\Delta(c )t_{l+1}$.

4. Following the proof of (6) of Theorem~2.1 in~\cite{Nishida-Soga}, we can prove that $c+\bu^k_m(c )$ is a $C^1$-function of $c$ for each $m,k$.
Therefore, Claim~2 yields $C^1$-regularity of $\bar{h}_\Delta$.
Let $v^{l+1}_n(c )$ be a solution of (\ref{373}) and fix $n$.
We show that the map $c\mapsto v^{l+1}_n(c )$ is a concave function for each $l+1\ge 1$.
Let $\xi^\ast$ be the minimizing velocity field for $v^{l+1}_{n}(c^\ast)$ with $c^\ast:=\theta c+(1-\theta)\tilde{c}$, $\theta\in[0,1]$.
Then, we have
\begin{eqnarray*}
&&v^{l+1}_n(c^\ast )-\{\theta v^{l+1}_n(c )+(1-\theta)v^{l+1}_n(\tilde{c})\}\\
&&\ge\theta E_{\mu(\cdot;\xi^\ast)}\Big[\sum_{0<k\le l+1}-(c^\ast-c)\xi^\ast{}^{k}_{m(\gamma^k)}\dt\Big]
+(1-\theta)E_{\mu(\cdot;\xi^\ast)}\Big[\sum_{0<k\le l+1}-(c^\ast-\tilde{c})\xi^\ast{}^{k}_{m(\gamma^k)}\dt\Big]\\
&&=0.
\end{eqnarray*}
Therefore, the map $c\mapsto v^{l+1}_n(c )/t_{l+1}$ is also a concave function and 
$$\bar{h}_\Delta(c)=-\lim_{l\to\infty}\frac{v^{l+1}_n(c )}{t_{l+1}}$$
 is a convex function.

5. Hereinafter $b_1,b_2,\ldots$ are positive constants independent of $\Delta$ and $c$.
For each $x\in\T$, we have $m$ such that $x\in[x_{m+1},x_{m+3})$.
Note that $\bar{v}^{2K}_{n_\ast+1}\le\bar{v}_\Delta(x,1)\le\bar{v}^{2K}_{n^\ast+1}$ with $(n_\ast,n^\ast)=(m,m+2)$ or $(n_\ast,n^\ast)=(m+2,m)$ and
\begin{eqnarray}\label{375-1}
\qquad \bar{v}^{2K}_{n_\ast+1}-\bar{v}(x_{n_\ast+1},1)-2r\dx\le \bar{v}_\Delta(x,1)-\bar{v}(x,1)\le \bar{v}^{2K}_{n^\ast+1}-\bar{v}(x_{n^\ast+1},1)+2r\dx.
\end{eqnarray}
\indent Let $x\in\T$ attain $\max_{y\in\T}(\bar{v}_\Delta(y,1)-\bar{v}(y,1))$ and let $n^\ast$ be defined in the above manner with this $x$.
Let $\gamma^\ast$ be a minimizing curve for $\bar{v}(x_{n^\ast+1},t)$.
Define $\xi$ as $\xi^k_m:=\gamma^\ast{}'(t_k)$.
Note that the $\eta^k(\gamma)$ defined by this $\xi$ satisfies $|\eta^k(\gamma)-\gamma^\ast(t_k)|\le b_1\dx$ for any $0\le k\le 2K$.
By the representation formulas and Proposition~\ref{eta}, we have
$$\bar{v}(x_{n^\ast+1},1)=\int_0^1L^{(c)}(\gamma^\ast(s),s,\gamma^\ast{}'(s))ds+\bar{v}(\gamma^\ast(0),0)+\bar{h}(c ),\qquad\qquad\qquad$$
\begin{eqnarray}\label{375-2}
&&\bar{v}^{2K}_{n^\ast+1}\le E_{\mu(\cdot;\xi)}\Big[\sum_{0<k\le 2K}L^{(c)}(\gamma^k,t_{k-1},\xi^k_{m(\gamma^k)})\dt +\bar{v}_\Delta(\gamma^0,0)\Big]+\bar{h}_\Delta(c )\\\nonumber
&&\le E_{\mu(\cdot;\xi)}\Big[\sum_{0<k\le 2K}L^{(c)}(\eta^k(\gamma),t_{k-1},\xi^k_{m(\gamma^k)})\dt +\bar{v}_\Delta(\eta^0(\gamma),0)\Big]+\bar{h}_\Delta(c )+b_2\sqrt{\dx}\\\nonumber
&&\le E_{\mu(\cdot;\xi)}\Big[\sum_{0<k\le 2K}L^{(c)}(\gamma^\ast(t_k),t_{k-1},\gamma^\ast{}'(t_k))\dt +\bar{v}_\Delta(\gamma^\ast(0),0)\Big]+\bar{h}_\Delta(c )+b_3\sqrt{\dx}\\\nonumber
&&\le \int^1_0L^{(c)}(\gamma^\ast(s),s,\gamma^\ast{}'(s))ds+\bar{v}_\Delta(\gamma^\ast(0),0)+\bar{h}_\Delta(c )+b_4\sqrt{\dx}.
\end{eqnarray}
Therefore, noting (\ref{375-1}), we have
$$\bar{v}_\Delta(x,1)-\bar{v}(x,1)\le \bar{v}_\Delta(\gamma^\ast(0),0)-\bar{v}(\gamma^\ast(0),0)+\bar{h}_\Delta(c )-\bar{h}(c )+b_5\sqrt{\dx}.$$
From the periodicity of $\bar{v}_\Delta,\bar{v}$ and the above choice of $x$ it follows that $(\bar{v}_\Delta(x,1)-\bar{v}(x,1))- (\bar{v}_\Delta(\gamma^\ast(0),0)-\bar{v}(\gamma^\ast(0),0))\ge 0$. Therefore, we obtain
$$-b_5\sqrt{\dx}\le \bar{h}_\Delta(c )-\bar{h}(c ).$$
\indent Let $x\in\T$ attain $\min_{y\in\T}(\bar{v}_\Delta(y,1)-\bar{v}(y,1))$ and let $n_\ast$ be defined in the above manner with this $x$.
Let $\xi^\ast$ be the minimizing velocity field for $\bar{v}^{2K}_{n_\ast+1}$.
Then, we have
\begin{eqnarray*}
\bar{v}^{2K}_{n_\ast+1}&=&E_{\mu(\cdot;\xi^\ast)}\Big[\sum_{0<k\le 2K}L^{(c)}(\gamma^k,t_{k-1},\xi^\ast{}^k_{m(\gamma^k)})\dt +\bar{v}_\Delta(\gamma^0,0)\Big]+\bar{h}_\Delta(c ),\\
&\ge& E_{\mu(\cdot;\xi^\ast)}\Big[\sum_{0<k\le 2K}L^{(c)}(\eta^k(\gamma),t_{k-1},\xi^\ast{}^k_{m(\gamma^k)})\dt +\bar{v}_\Delta(\eta^0(\gamma),0)\Big]+\bar{h}_\Delta(c )-b_6\sqrt{\dx}.
\end{eqnarray*}
Let $\eta_\Delta(\gamma)$ be the linear interpolation of $\eta^k(\gamma)$.
Note that $\eta_\Delta(\gamma)'(t)=\xi^\ast{}^k_{m(\gamma^k)}$ for $t\in(t_{k-1},t_k)$.
For each $\gamma$ we have
\begin{eqnarray}\label{375-3}
\bar{v}(x_{n_\ast+1},1)&\le&\int_0^1L^{(c)}(\eta_\Delta(\gamma)(s),s,\eta_\Delta(\gamma)'(s))ds+\bar{v}(\eta_\Delta(\gamma)(0),0)+\bar{h}(c )\\\nonumber
&\le&\sum_{0<k\le 2K}L^{(c)}(\eta^k(\gamma),t_{k-1},\xi^\ast{}^k_{m(\gamma^k)})\dt +\bar{v}(\eta^0(\gamma),0)+\bar{h}(c )+b_7\dx.
\end{eqnarray}
Therefore, noting (\ref{375-1}), we have
\begin{eqnarray*}
\bar{v}_\Delta(x,1)-\bar{v}(x,1)\ge E_{\mu(\cdot;\xi^\ast)}\Big[\bar{v}_\Delta(\eta^0(\gamma),0)-\bar{v}(\eta^0(\gamma),0)\Big]+\bar{h}_\Delta(c )-\bar{h}(c )-b_8\sqrt{\dx}.
\end{eqnarray*}
From the periodicity of $\bar{v}_\Delta,\bar{v}$ and the above choice of $x$ it follows that  $(\bar{v}_\Delta(x,1)-\bar{v}(x,1))-(\bar{v}_\Delta(\eta^0(\gamma),0)-\bar{v}(\eta^0(\gamma),0))\le 0$ for all $\gamma$. Thus, we obtain
$$\bar{h}_\Delta(c )-\bar{h}(c )\le b_8\sqrt{\dx}.$$
\end{proof}
\subsection{Convergence of Periodic Solutions}

We prove that for space-time periodic solutions the difference solutions converge to the exact ones up to a subsequence.
Note that viscosity solutions and entropy solutions with space-time periodicity are not unique with respect to $c$ in general.
The selection problem in finite difference approximation remains open.
It is also challenging to investigate details of the convergence even in the case where the uniqueness holds.
We will make some progress with this issue in the next section.
\begin{Thm}
There exists a sequence $\Delta=(\dx,\dt)\to0$ for which $\{\bar{v}_\Delta^{(c)}\}$ and $\{\bar{u}_\Delta^{(c)}\}$ converge to a $\Z^2$-periodic viscosity solution $\bar{v}$ of (\ref{HJ2}) with $h(c)=\bar{h}(c )$ and to a $\Z^2$-periodic entropy solution  $\bar{u}=\bar{v}_x$ of (\ref{CL2}), respectively:
$$\sup_{t\in\T}\norm \bar{v}_\Delta^{(c)}(\cdot,t)-\bar{v}(\cdot,t)\norm_{C^0}\to0,\,\,\, \,\,\,\,\sup_{t\in\T}\norm \bar{u}_\Delta^{(c)}(\cdot,t)-\bar{u}(\cdot,t)\norm_{L^1(\T)}\to0.$$
\end{Thm}
\begin{proof}[{\bf Proof.}]
If necessary, we add a constant so that $\bar{v}^{(c)}_\Delta(\cdot,0)$ is bounded by $r$.
Then, $\{\bar{v}_\Delta^{(c)}(\cdot,0)\}$ is a family of functions that are uniformly bounded and equicontinuous.
We have a convergent subsequence, still denoted by $\bar{v}_\Delta^{(c)}(\cdot,0)$: $\bar{v}_\Delta^{(c)}(\cdot,0)\to \bar{v}^0$.
Let $\bar{v}$ be the viscosity solution of (\ref{HJ}) with $v^0=\bar{v}^0$ and $h(c )=\bar{h}(c )$.
Then, we have a minimizing curve such that
$$\bar{v}(x_n,t_{l+1})=\int^{t_{l+1}}_0L^{(c)}(\gamma^\ast(s),s,\gamma^\ast{}'(s))ds+\bar{v}^0(\gamma^\ast(0))+\bar{h}(c )t_{l+1}.$$
By an estimate similar to (\ref{375-2}), we have
$$\bar{v}_\Delta(x_n,t_{l+1})\le \int^{t_{l+1}}_0L^{(c)}(\gamma^\ast(s),s,\gamma^\ast{}'(s))ds+\bar{v}_\Delta(\gamma^\ast(0),0)+\bar{h}_\Delta(c )t_{l+1}+b_1\sqrt{\dx}.$$
Since $\bar{h}_\Delta(c )\to\bar{h}(c )$, we obtain
$$\limsup_{\Delta\to0}\{\bar{v}_\Delta^{(c)}(x_n,t_{l+1})-\bar{v}(x_n,t_{l+1})\}\le0,$$
which is uniform with respect to $(x_n,t_{l+1})\in\T^2$.
By an estimate similar to (\ref{375-3}),  we obtain
$$\liminf_{\Delta\to0}\{\bar{v}_\Delta^{(c)}(x_n,t_{l+1})-\bar{v}(x_n,t_{l+1})\}\ge0,$$
which is uniform with respect to $(x_n,t_{l+1})\in\T^2$.
Therefore, we conclude that $\bar{v}_\Delta^{(c)}\to\bar{v}$ uniformly on $\T^2$ and $\bar{v}$ is $\Z^2$-periodic due to the periodicity of $\bar{v}_\Delta$.

Through reasoning similar to the proof of Theorem~2.8 in~\cite{Soga2}, it follows that $\bu_\Delta^{(c)}:=(\bar{v}^{(c)}_\Delta{})_x$ converges to $\bu=\bar{v}_x$ pointwise almost everywhere in $\T^2$, where $\{\bar{v}_\Delta^{(c)}\}$ is the convergent subsequence above.
Hence, we have $\norm \bar{u}_\Delta^{(c)}(\cdot,t)-\bu(\cdot,t)\norm_{L^1(\T)}\to0$ for each $t$.
Through reasoning similar to the proof of Proposition~2.14 in~\cite{Nishida-Soga}, it follows that $\bu_\Delta^{(c)}$ satisfies
$$\norm \bu^{(c)}_\Delta(\cdot,t_k)-\bu^{(c)}_\Delta(\cdot,t_{k'})\norm_{L^1(\T)}\le b_2|t_k-t_{k'}|$$
with a constant $b_2$ independent of $k$, $k'$, $c$, and $\Delta$.
Therefore, $\bu\in Lip(\T;L^1(\T))$ with the Lipschitz constant $b_2$.
Thus, we have demonstrated the theorem.
\end{proof}
\setcounter{section}{3}
\setcounter{equation}{0}
\section{Error Estimates}
We show error estimates for entropy solutions of initial value problems  and for  $\Z^2$-periodic entropy solutions in the special case where they are associated with KAM tori. The latter is a rigorous result on finite difference approximation of KAM tori. We refer to \cite{Bessi} for an error estimate for $\Z^2$-periodic entropy solutions associated with KAM tori in the vanishing viscosity method.
\subsection{Error Estimates for Initial Value Problem}
The following theorem provides error estimates for the initial value problem.
\begin{Thm}\label{Main-Error} Let $\Delta=(\dx,\dt)$ satisfy the conditions in Theorem~\ref{SVA} and Proposition~\ref{entropy-condition}.
Let $u$ be the entropy solution of (\ref{CL}) and $u_\Delta$ be given by the difference solution of (\ref{2CL-Delta}).
Then, the following hold:
\begin{enumerate}
\item For any $T\in(0,\infty)$, for each $t\in(0,T]$, and independent of the initial data, there exists a constant $\beta_4(t)>0$  for which
$$\norm u_\Delta(\cdot,t)-u(\cdot,t)\norm_{L^1(\T)}\le \beta_4(t)\dx^{\frac{1}{4}}.$$
In particular, if $u^0$ is rarefaction-free, then there exists a constant $\beta_5>0$ for which
$$\sup_{0\le t\le T}\norm u_\Delta(\cdot,t)-u(\cdot,t)\norm_{L^1(\T)}\le \beta_5\dx^{\frac{1}{4}}.$$
\item If $u$ is Lipschitz in $\T\times[0,T]$, then there exists a constant $\beta_6>0$ for which
$$\sup_{(x,t)\in\T\times[0,T]}| u_\Delta(x,t)-u(x,t)|\le \beta_6\dx^{\frac{1}{4}}.$$
\end{enumerate}
\end{Thm}
\begin{proof}[{\bf Proof.}]
 1. Let $v_\Delta$ and $v$ correspond to $u_\Delta$ and $u$, respectively.
 By Theorem~\ref{SVA}, for all $t\in[0,T]$ and all initial data, we have
 \begin{eqnarray}\label{414141}
 \norm v_\Delta(\cdot,t)-v(\cdot,t)\norm_{C^0}\le \beta_2\sqrt{\dx}.
 \end{eqnarray}
 By  Proposition~\ref{entropy-condition}, for each $t\in[\dt,T]$ and all initial data, we have
 \begin{eqnarray}\label{3one-sided}
 \frac{u_\Delta(x_{m+2},t)-u_\Delta(x_m,t)}{2\dx}\le E^{k(t)}_\Delta.
 \end{eqnarray}
 Since $u_\Delta(\cdot,t)$ has zero mean, we have
 \begin{eqnarray*}
 \sum_{m;k(t)}\{u_\Delta(x_{m+2},t)-u_\Delta(x_m,t)\}
 &=&\sum_{m:+}\{u_\Delta(x_{m+2},t)-u_\Delta(x_m,t)\}\\
 &&+\sum_{m:-}\{u_\Delta(x_{m+2},t)-u_\Delta(x_m,t)\}\\
 &=&0,
 \end{eqnarray*}
where $\sum_{m:+}$ (resp. $\sum_{m:-}$) stands for the summation with respect to $m$ for which $u_\Delta(x_{m+2},t)-u_\Delta(x_m,t)\ge0$ (resp. $u_\Delta(x_{m+2},t)-u_\Delta(x_m,t)<0$).
Hence, it follows from (\ref{3one-sided}) that the total variation of $u_\Delta(\cdot,t)$ on $\T$ is bounded:
\begin{eqnarray}\label{3total-variation}
\sum_{m;k(t)}|u_\Delta(x_{m+2},t)-u_\Delta(x_m,t)|=2\sum_{m:+}\{u_\Delta(x_{m+2},t)-u_\Delta(x_m,t)\}\le 2E^{k(t)}_\Delta.
\end{eqnarray}
\indent For any $\varepsilon>0$, there exists $\tilde{\Delta}=(\tilde{\dx},\tilde{\dt})$ such that
$$\norm u_{\tilde{\Delta}}(\cdot,t)-u(\cdot,t)\norm_{L^1(\T)}\le \varepsilon.$$
In particular, we take such a $\tilde{\Delta}=(\tilde{\dx},\tilde{\dt})$ that satisfies $\tilde{\dt}/\tilde{\dx}=\dt/\dx$, $\tilde{\dx}\le (\beta^{-1}_2\varepsilon)^4$, and $\dx/\tilde{\dx}=3^p$ for some $p\in\N$.
The last relation guarantees that the points of discontinuity of $u_\Delta$ are also those of $u_{\tilde{\Delta}}$.  Then, we have
\begin{eqnarray}\nonumber
\norm u_\Delta(\cdot,t)-u(\cdot,t)\norm_{L^1(\T)}&\le& \norm u_{\tilde{\Delta}}(\cdot,t)-u_{\Delta}(\cdot,t)\norm_{L^1(\T)}+\varepsilon,\\\label{41-1}
\norm v_{\tilde{\Delta}}(\cdot,t)-v_\Delta(\cdot,t)\norm_{C^0}&\le& \beta_2\sqrt{\dx} +\beta_2\sqrt{\tilde{\dx}}\le 2\beta_{2}\sqrt{\dx}.
\end{eqnarray}
Now we estimate $\norm u_{\tilde{\Delta}}(\cdot,t)-u_{\Delta}(\cdot,t)\norm_{L^1(\T)}$.
We introduce $w_\Delta:=u_{\tilde{\Delta}}(\cdot,t)-u_{\Delta}(\cdot,t)$, $\tilde{w}_\Delta:=v_{\tilde{\Delta}}(\cdot,t)-v_{\Delta}(\cdot,t)$ and $\tilde{k}(t):=3^pk(t)$.
Let $x_m\in\tilde{\dx}\Z$ and set $x_{m_0}:=0$ for $\tilde{k}(t)$ even or $x_{m_0}:=\tilde{\dx}$ for $\tilde{k}(t)$ odd.
We divide $\tilde{\dx}\Z$ according to the sign of $w_\Delta$.
That is, $I_1,I_2,\ldots,I_{n+1}$ are defined as
 \begin{eqnarray*}
 I_1&:=&\{x_{m_0},x_{m_{0}+2},\cdots,x_{m_{1}} \}\mbox{ on which $w_\Delta(x)\ge0$ (or $<0$)},\\
 I_2&:=&\{x_{m_{1}+2},x_{m_{1}+4},\cdots,x_{m_{2}} \}\mbox{ on which $w_\Delta(x)<0$ (or $\ge0$)},\\
I_3&:=&\{x_{m_2+2},x_{m_{2}+4},\cdots,x_{m_{3}} \}\mbox{ on which $w_\Delta(x)\ge0$ (or $<0$)},\\
 && \cdots,\\
  I_n&:=&\{x_{m_{n-1}+2},x_{m_{n-1}+4},\cdots,x_{m_{n}} \}\mbox{ on which $w_\Delta(x)<0$ (or $\ge0$)},\\
 I_{n+1}&:=&\{x_{m_{n}+2},x_{m_{n}+4},\cdots,x_{m_{0}}+1 \} \mbox{ on which $w_\Delta(x)\ge0$ (or $<0$)},
 \end{eqnarray*}
where $n$ is even and $x_{m_n}\le x_{m_0}+1-2\tilde{\dx}$.
We then redefine $I_1$ as
 $$I_1:=\{x_{m_0},x_{m_{0}+2},\cdots,x_{m_{1}} \}\mbox{ with $x_{m_0}:= x_{m_n+2}-1$}.$$
Note that $w_\Delta(x)\ge 0$ (or $<0$) on $I_1$.
Setting $|I_1|:=x_{m_1}-x_{m_0}+2\tilde{\dx}$ and $|I_j|:=x_{m_j}-x_{m_{j-1}+2}+2\dx$ for $j>1$,
we have $\sum_{j=1}^n|I_j|=1$.
For each $I_j$ on which $w_\Delta(x)\ge 0$ (resp. $<0$), we have a $y^j\in I_j$ for which $w_\Delta(x)$ takes the maximum (resp. minimum) within $I_j$.
Suppose that $w_\Delta(x)\ge 0$ on $I_1$.
In the other case, the argument is  parallel.
Note that
 $$\norm w_\Delta(x)\norm_{L^1(\T)}=\sum_{j=1}^{n/2}\left\{\sum_{x\in I_{2j-1}}w_\Delta(x)\cdot2\tilde{\dx}-\sum_{x\in I_{2j}}w_\Delta(x)\cdot2\tilde{\dx}\right\}.$$
Introducing $J:=\{j\,|\,0\le j\le n/2,\,\,\,\max\{  |I_{2j-1}|,|I_{2j}| \}<\dx^{1/4} \}$ and $\tilde{J}:=\{j\,|\,0\le j\le n/2,\,\,\,\max\{  |I_{2j-1}|,|I_{2j}| \}\ge\dx^{1/4} \}$, we have $\sharp \tilde J\cdot \dx^{1/4}\le 1$ and $\sharp \tilde{J}\le \dx^{-1/4}$.
Therefore, noting (\ref{3total-variation}) and (\ref{41-1}) as well as $w_\Delta=(\tilde{w}_\Delta)_x$, we obtain
\begin{eqnarray*}
\norm w_\Delta(x)\norm_{L^1(\T)}&= &\sum_{j\in J} \left\{\sum_{x\in I_{2j-1}}w(x)\cdot2\tilde{\dx}-\sum_{x\in I_{2j}}w_\Delta(x)\cdot2\tilde{\dx}     \right\}\\
&&+\sum_{j\in\tilde{J}}\left\{\sum_{x\in I_{2j-1}}w_\Delta(x)\cdot2\tilde{\dx}-\sum_{x\in I_{2j}}w_\Delta(x)\cdot2\tilde{\dx}     \right\}\\
&\le& \sum_{j\in J}|w_\Delta(y^{2j-1})-w_\Delta(y^{2j})|\dx^{\frac{1}{4}}\\
&&+\sum_{j\in \tilde{J}}\Big[\{ \tilde{w}_\Delta(x_{m_{2j-1}}+\tilde{\dx})-\tilde{w}_\Delta(x_{m_{2j-2}+2}-\tilde{\dx})\}\\
&&-\{ \tilde{w}_\Delta(x_{m_{2j}}+\tilde{\dx})-\tilde{w}_\Delta(x_{m_{2j-1}+2}-\tilde{\dx})\}  \Big]\\
&\le& (2E^{\tilde{k}(t)}_{\tilde{\Delta}}+2E^{k(t)}_\Delta)\dx^{\frac{1}{4}}+\sharp\tilde{J}\cdot4\cdot2\beta_2\sqrt{\dx}\\
&\le& 4E^{k(t)}_\Delta\dx^{\frac{1}{4}}+8\beta_2\dx^{\frac{1}{4}},
\end{eqnarray*}
Since $\varepsilon$ is arbitrary, we conclude that
$$\norm u_{\Delta}(\cdot,t)-u(\cdot,t)\norm_{L^1(\T)}\le (4E^{k(t)}_\Delta+8\beta_2)\dx^{\frac{1}{4}}.$$
If $u^0$ is rarefaction-free, we have $M>0$ such that $E^{k(t)}_\Delta\le\max\{M,E^\ast\}$ for all $0\le t\le T$.

2. Fix $t\in[0,T]$ arbitrarily.
By (\ref{414141}), for any $x,x'\in\T$ we have
\begin{eqnarray*}
|\int^x_{x'} u_\Delta(y,t)-u(y,t)dy|\le 2 \beta_2\sqrt{\dx}.
\end{eqnarray*}
From (\ref{3one-sided}) it follows that $\norm\tilde{u}_\Delta-u_\Delta(\cdot,t)\norm_{L^1(\T)}\le b_1\dx$,  where $\tilde{u}_\Delta(x)$ denotes the linear interpolation of $u^{k(t)}_m$ with respect to the space variable.
Hence, setting $w_\Delta:=\tilde{u}_\Delta-u(\cdot,t)$, for all $x,x'\in\T$ we have
\begin{eqnarray}\label{4aaa}
|\int^x_{x'}w_\Delta(y)dy|\le b_2\sqrt{\dx}.
\end{eqnarray}
Since $u$ is Lipschitz, $w_\Delta$ still satisfies the one-sided Lipschitz condition
$$\frac{w_\Delta(x_1)-w_\Delta(x_2)}{x_1-x_2}\le b_3.$$
Note that $w_\Delta$ does not necessarily satisfy any Lipschitz condition, because $\tilde{u}_\Delta$ does not necessarily satisfy any Lipschitz condition.
Suppose that $|w_\Delta(\bar{x})|> b_{4}\dx^{\frac{1}{4}}$ with $(b_{4})^2/(4b_2)>b_3$ for some $\bar{x}$. Let $I\ni\bar{x}$ be a connected interval on whose boundary we have $|w_\Delta(x)|=\frac{b_{4}}{2}\dx^{\frac{1}{4}}$.
By (\ref{4aaa}), we find that
$$|I|\le\frac{2b_2}{b_{4}}\dx^{\frac{1}{4}}.$$
If $w_\Delta(\bar{x})>0$ (resp. $<0$), and with the left (resp. right) boundary of $I$ denoted by $x$, we have
$$\frac{w_\Delta(\bar{x})-w_\Delta(x)}{\bar{x}-x}\ge \frac{(b_{4})^2}{4b_2}>b_3\quad\left( \mbox{resp. } \frac{w_\Delta(x)-w_\Delta(\bar{x})}{x-\bar{x}}\ge \frac{(b_{4})^2}{4b_2}>b_3 \right),$$
which is a contradiction.
Therefore, we obtain
$$\norm w_\Delta\norm_{C^0}\le  b_{4}\dx^{\frac{1}{4}}.$$
Since $|u_\Delta(x,t)-u(x,t)|=|u^{k(t)}_m-u(x,t)|\le|u^{k(t)}_m-u(x_m,t)|+b_{5}\dx=|w_\Delta(x_m)|+b_{5}\dx$, we have demonstrated the theorem.
\end{proof}
\subsection{Error Estimate for KAM Tori}
Let $\bu^{(c)}=\bar{v}_x^{(c)}$ be a $\Z^2$-periodic entropy solution of the $C^1$-class.
We remark the relationship between such a $\bu^{(c)}$ and Hamiltonian dynamics.
Consider the time-1 map $f:\T\times\R\to\T\times\R$ of the Hamiltonian flow generated by the flux function $H(x,t,p)$ with the initial time equal to zero.
Then, $\{(x,c+\bu^{(c)}(x,0))\,|\,x\in\T\}\cong\T$ is a smooth invariant torus of $f$.
According to the classical result of Poincar\'e,  there exists a rotation number $\omega_1$.
Let us regard the nonautonomous Hamiltonian dynamics generated by $H(x,t,p)$ as the autonomous dynamics generated by $\mathcal{H}(q_1,q_2,p_1,p_2):=p_2+H(q_1,q_2,p_1)$ in the extended phase space $\T^2\times\R^2$.
We define
$$\I(\bu^{(c)}):=\{(q,g(q))\,|\,q=(q_1,q_2)\in\T^2\}\cong\T^2,$$
where $g(q):=(c+\bu^{(c)}(q_1,q_2),\bar{h}(c )-H(q_1,q_2,c+\bu^{(c)}(q_1,q_2))$.
Then, $\I(\bu^{(c)})$ is a smooth invariant torus of the Hamiltonian flow $\varphi^s_\mathcal{H}$ generated by $\mathcal{H}$.
Let  $C(s):=(\gamma^\ast(s),s)$ be the characteristic curves of $\bu^{(c)}$, which satisfy $\gamma^\ast{}'(s)=H_p(\gamma^\ast(s),s,c+\bu^{(c)}(\gamma^\ast(s),s))$ for $s\in\R$.
The dynamics of the reduced characteristic curves $C^\ast(s):=C(s)\mod1=(\gamma^\ast(s)\mod1,s\mod1)$ and that of the trajectories on $\I(\bu^{(c)})$ are identical; namely, for all $s\in\R$ we have
$$\varphi^s_\mathcal{H}(C^\ast(0),g(C^\ast(0)))=(C^\ast(s),g(C^\ast(s))).$$
According to the classical result of Poincar\'e, $C(s)/s$ converges to  $\omega=(\omega_1,1)\in\R^2$  independently of $C(0)$ as $|s|\to\infty$.
This $\omega$ is called a rotation vector of $\I(\bu^{(c)})$.
If the rotation vector is irrational, each trajectory starting from a point of $\I(\bu^{(c)})$ is dense on $\I(\bu^{(c)})$.
Therefore, we can obtain information on $\bu^{(c)}$ from merely one characteristic curve, which is the crucial fact in the subsequent argument.
Approximation of $\bu^{(c)}$ leads to that of the invariant torus $\I(\bu^{(c)})$.

Now we consider a special case where $\I(\bu^{(c)})$ is a KAM torus.
We say that $c$ is associated with a KAM torus if $\bu^{(c)}$ is $C^1$ and the dynamics of $C^\ast(s)$ is $C^1$-conjugate to that of a linear flow on $\T^2$ with a Diophantine rotation vector; namely, there exists a diffeomorphism $F:\T^2\to\T^2$ such that
$$C^\ast(s)=F(\omega s+\theta),$$
where $\theta\in\R$ depends on $C^\ast(0)$ and $\omega\in\R^2$ satisfies the $\nu,\tau$-Diophantine condition
$$|\omega_1 z_1+\omega_2z_2|\ge\nu \norm z\norm_1^{-\tau}\mbox{\quad for all $z\in\Z^2\setminus\{0\}$}.$$
If $c$ is associated with a KAM torus, then $\bu^{(c)}$ is the unique $\Z^2$-periodic entropy solution of (\ref{CL2}) with that $c$.
Regarding the existence of a value $c$ associated with a KAM torus, we refer to the classical KAM theory for the autonomous Hamiltonian systems with two degrees of freedom generated by the above $\mathcal{H}(q_1,q_2,p_1,p_2)$.
We remark that with additional assumptions the classical KAM theory under R\"ussmann's nondegenerate condition (e.g., see~\cite{Sevryuk}) works for such a degenerate $\mathcal{H}$ in $(p_1,p_2)$.
The following theorem provides error estimates for KAM tori.
\begin{Thm}\label{error-KAM} Let $\Delta=(\dx,\dt)$ satisfy the conditions in Theorem~\ref{SVA}, Proposition~\ref{entropy-condition}, and Theorem~\ref{Main-stability}.
Suppose that $c$ is associated with a KAM torus.
Let $\bar{v}^{(c)}$ be a $\Z^2$-periodic viscosity solution such that $\bar{v}^{(c)}_x=\bu^{(c)}$.
Then, the space-time periodic difference solutions $\bar{v}^{(c)}_\Delta$ and $\bu^{(c)}_\Delta$ satisfy
\begin{eqnarray*}
\sup_{(x,t)\in\T^2}| \bar{v}_\Delta^{(c)}(x,t)-\bar{v}^{(c)}(x,t)|\le\beta_7\dx^{\frac{1}{2(1+\tau)}},\quad\sup_{(x,t)\in\T^2} |\bar{u}_\Delta^{(c)}(x,t)-\bar{u}^{(c)}(x,t)|\le\beta_8\dx^{\frac{1}{4(1+\tau)}},
\end{eqnarray*}
where $\beta_7$ and $\beta_8$ are independent of $\Delta$.
\end{Thm}
\begin{proof}[{\bf Proof.}] Let $\bar{v}^{(c)}_\Delta$ be a periodic difference solution.
In what follows, we omit the superscript $(c)$ in $\bar{v}^{(c)}$, $\bu^{(c)}$, etc.
Fix $t\in\T$ arbitrarily.
By adding a constant to $\bar{v}_\Delta$ if necessary, we have $\bar{v}_\Delta(x,t)-\bar{v}(x,t)\le0$ for all $x\in\T$ and $\bar{v}_\Delta(x^\ast,t)-\bar{v}(x^\ast,t)=0$ for some $x^\ast\in\T$.
Then, we have $n^\ast$ and $l$ such that
$$0=\bar{v}_\Delta(x^\ast,t)-\bar{v}(x^\ast,t)\le \bar{v}^l_{n^\ast+1}-\bar{v}(x_{n^\ast+1},t_l)+b_1\dx,$$
where $|x^\ast-x_{n^\ast+1}|\le2\dx$ and $t\in[t_l,t_{l+1})$.
For any $j\in\N$, we have a minimizing curve $\gamma^\ast$ such that
\begin{eqnarray*}
\dis \bar{v}(x_{n^\ast+1},t_l)=\int^{t_l}_{-j+t_l}L^{(c)}(\gamma^\ast(s),s,\gamma^\ast{}'(s))ds+\bar{v}(\gamma^\ast(-j+t_l),-j+t_l)+\bar{h}(c )j, \\
\bar{v}^l_{n^\ast+1}\le\int^{t_l}_{-j+t_l}L^{(c)}(\gamma^\ast(s),s,\gamma^\ast{}'(s))ds+\bar{v}_\Delta(\gamma^\ast(-j+t_l),-j+t_l)+\bar{h}_\Delta(c )j+b_2\sqrt{\dx}j,
\end{eqnarray*}
where we use an estimate similar to (\ref{375-2}).
Hence, with Claim~5 of Theorem~\ref{effective} we obtain
\begin{eqnarray*}
0&\le&\bar{v}^l_{n^\ast+1}-\bar{v}(x_{n^\ast+1},t_l)+b_1\dx\\
&\le& \bar{v}_\Delta(\gamma^\ast(-j+t_l),-j+t_l)-\bar{v}(\gamma^\ast(-j+t_l),-j+t_l)\\
&&+(\bar{h}_\Delta(c )-\bar{h}(c ))j+b_2\sqrt{\dx}j+b_1\dx\\
&\le&\bar{v}_\Delta(\gamma^\ast(-j+t),t)-\bar{v}(\gamma^\ast(-j+t),t)+b_3\sqrt{\dx}j.
\end{eqnarray*}
Since $\bar{v}_\Delta(x,t)-\bar{v}(x,t)\le0$ for all $x\in\T$, we then obtain
$$|\bar{v}_\Delta(\gamma^\ast(-j+t),t)-\bar{v}(\gamma^\ast(-j+t),t)|\le b_3\sqrt{\dx}j.$$
Since $C^\ast(-s+t):=(\gamma^\ast(-s+t),-s+t)\mod1$ is a reduced characteristic curve, we have $C^\ast(-s+t)=F(\omega (-s+t)+\theta)$.
From \cite{Dumas} and \cite{BGW} it follows that the set
$$\mathcal{N}_\varepsilon:=\{\theta+\omega(-s+t)\mod1\,|\,0\le s\le\frac{b_4}{\varepsilon^\tau}\}$$
is $\varepsilon$-dense on $\T^2$; namely,
$$\bigcup_{\zeta\in\mathcal{N}_\varepsilon}B_\varepsilon(\zeta)=\T^2,$$
where $B_\varepsilon(\zeta)=\{\tilde{\zeta}\in\T^2\,|\,\norm\tilde{\zeta}-\zeta\norm_1\le\varepsilon\}$.
We define $\mathcal{T}:=F^{-1}(\T\times\{t\})$ and $X:=(x,t)$ for $x\in\T$.
For each $X$, we have $\zeta\in\mathcal{N}_\varepsilon\cap\mathcal{T}$ such that $\norm\tilde{X}-\zeta\norm_1\le\varepsilon$ with $\tilde{X}:=F^{-1}(X)$ and such that $\zeta=\omega(-s^\ast+t)+\theta\mod1$ with some $0\le s^\ast\le\frac{b_4}{\varepsilon^\tau}$.
Note that $s^\ast$ must be an integer, because $F(\zeta)=C^\ast(-s^\ast+t)\in\T\times\{t\}$ and $-s^\ast+t\mod 1=t$.
Hence, setting $s^\ast=j$, we have
$$\norm X-C^\ast(-j+t)\norm_1=\norm F(\tilde{X})-F(\zeta)\norm_1\le\norm DF\norm_{op}\varepsilon.$$
Therefore, for all $x\in\T$ we obtain
\begin{eqnarray*}
|\bar{v}_\Delta(x,t)-\bar{v}(x,t)|&\le&|\bar{v}_\Delta(F(\tilde{X}))-\bar{v}_\Delta(F(\zeta))|+|\bar{v}_\Delta(F(\zeta))-\bar{v}(F(\zeta))|\\
&&+|\bar{v}(F(\zeta))-\bar{v}(F(\tilde{X}))|\\
&\le& b_5\varepsilon+b_3\sqrt{\dx}j+b_5\varepsilon\\
&\le& b_6(\frac{\sqrt{\dx}}{\varepsilon^\tau}+\varepsilon).
\end{eqnarray*}
Taking $\varepsilon=\dx^{\frac{1}{2(1+\tau)}}$, for all $x\in\T$ we have
$$|\bar{v}_\Delta(x,t)-\bar{v}(x,t)|\le 2b_6\dx^{\frac{1}{2(1+\tau)}}.$$
Note that $b_6$ is independent of the choice of $t$.
For $\bu_\Delta=(\bar{v}_\Delta)_x$, $\bar{u}=\bar{v}_x$, and all $x,x'\in\T$, we have
\begin{eqnarray*}
|\int^x_{x'}\bar{u}_\Delta(y,t)-\bar{u}(y,t)dy|\le 4b_6\dx^{\frac{1}{2(1+\tau)}}.
\end{eqnarray*}
Since  $\bu_\Delta$ satisfies the one-sided Lipschitz condition, we have $\norm\tilde{\bu}_\Delta-\bu_\Delta(\cdot,t)\norm_{L^1(\T)}\le b_7\dx$, where $\tilde{\bu}_\Delta(x)$ denotes the linear interpolation of $\bu^l_m$ with respect to the space variable.
Setting $w_\Delta:=\tilde{\bu}_\Delta-\bu(\cdot,t)$, for all $x,x'\in\T$ we have
\begin{eqnarray}\label{4aa}
|\int^x_{x'}w_\Delta(y)dy|\le b_8\dx^{\frac{1}{2(1+\tau)}}.
\end{eqnarray}
Since $\bu$ is $C^1$, we know that $w_\Delta$ still satisfies the one-sided Lipschitz condition
$$\frac{w_\Delta(x_1)-w_\Delta(x_2)}{x_1-x_2}\le b_9.$$
Suppose that $|w_\Delta(\bar{x})|>b_{10}\dx^{\frac{1}{4(1+\tau)}}$ with $(b_{10})^2/(4b_8)>b_9$ for some $\bar{x}$.
Let $I\ni\bar{x}$ be a connected interval on whose boundary we have $|w_\Delta(x)|=\frac{b_{10}}{2}\dx^{\frac{1}{4(1+\tau)}}$.
By (\ref{4aa}), we find that
$$|I|\le\frac{2b_8}{b_{10}}\dx^{\frac{1}{4(1+\tau)}}.$$
If $w_\Delta(\bar{x})>0$ (resp. $<0$), and with the left (resp. right) boundary of $I$ denoted by $x$, we have
$$\frac{w_\Delta(\bar{x})-w_\Delta(x)}{\bar{x}-x}\ge \frac{(b_{10})^2}{4b_8}>b_9\quad\left(\mbox{resp. } \frac{w_\Delta(x)-w_\Delta(\bar{x})}{x-\bar{x}}\ge \frac{(b_{10})^2}{4b_8}>b_9 \right),$$
which is a contradiction.
Therefore, we obtain
$$\norm w_\Delta\norm_{C^0}\le  b_{10}\dx^{\frac{1}{4(1+\tau)}}.$$
Since $|\bu_\Delta(x,t)-\bu(x,t)|=|\bu^l_m-\bu(x,t)|\le|\bu^l_m-\bu(x_m,t)|+b_{11}\dx=|w_\Delta(x_m)|+b_{11}\dx$, we have demonstrated the theorem.
\end{proof}
The point of our numerical approximation of KAM tori is that the embedding of each KAM torus is connected to a certain classical solution of the PDEs (\ref{CL2}) and (\ref{HJ2}), which are then solved numerically.
Note that the existence of such a classical solution is assumed in our argument.
The regularity criterion of solutions to (\ref{CL2}) and (\ref{HJ2}) under (A1)--(A4) remains an important open problem.
An estimate of the error between $\bar{u}^{(c)}_\Delta$ and $\bar{u}^{(c)}$ without the Diophantine condition or without the condition $\bar{u}^{(c)}\in C^1$ also remains open.
The latter is particularly interesting in the context of a rigorous treatment of numerical approximations of Aubry-Mather sets.

Finally, we describe in brief the idea of another numerical approach to KAM tori, which is based on the so-called a posteriori KAM theorem.
Let $f$ be the time-1 map given at the beginning of this subsection.
If there exists a smooth embedding $U^\ast:\T\to\T\times\R$ which satisfies the functional equation
\begin{eqnarray}\label{ap-KAM}
f\circ U(q)=U\circ T_{\omega_1}(q)\mbox{ \,\,\, for all $q\in\T$},
\end{eqnarray}
where $\omega_1\in\R$ and $T_{\omega_1}(q):=q+\omega_1$, then $U^\ast(\T)$ is a smooth invariant torus of $f$ on which the dynamics is $C^1$-conjugate to that of $T_{\omega_1}$ on $\T$.
The standard classical KAM theory leads to the fact that, if $\omega_1$ is a Diophantine number, a unique such $U^\ast$ exists under certain conditions on $f$.
The idea of the a posteriori KAM theorem is stated below.
We regard (\ref{ap-KAM}) as $\mathcal{F}(U):=f\circ U-U\circ T_{\omega_1}=0$ in a certain family $W$ of smooth mappings: $\T\to\T\times\R$, where $\mathcal{F}:W\to W$.
\medskip
\noindent{\bf Idea of a posteriori KAM Theorem.}
{\it Suppose that there exists $U^0\in W$ such that $\mathcal{F}(U^0)$ is close to $0$ in the norm of $W$.
Then, there exists unique $U^\ast$ such that $\mathcal{F}(U^\ast)=0$ and $\norm U^0-U^\ast\norm_W\le C\norm\mathcal{F}(U^0)\norm_W$.}
\medskip
\noindent In fact, this idea has been justified with the Diophantine condition of $\omega_1$ in many studies.
The a posteriori KAM theorem describes both the existence of KAM tori and their numerical approximation, since a suitable $U^0$ can be numerically constructed through Newton's method.
Moreover, the a posteriori KAM theorem can be successfully applied to find the magnitude of perturbation at which the classical KAM theory breaks down.
We point to \cite{de la Llave} for a nice presentation and survey with plenty of references for the a posteriori KAM theorem and its applications.
The a posteriori KAM theorem provides no information on the situation after the classical KAM theory breaks down.
On the other hand, the weak KAM theory still guarantees the existence of Aubry-Mather sets with arbitrary rotation numbers.
It will be an important contribution to recast the results of the classical KAM theory or the weak KAM theory in terms of the other, including a more detailed comparison of our result with those based on the a posteriori KAM theorem.

\end{document}